\documentclass[reqno, 11pt]{amsart}
\usepackage{amssymb}
\usepackage[colorlinks=true]{hyperref}
\usepackage{geometry}
\usepackage{mathabx}

\newtheorem{theorem}{Theorem}[section]
\newtheorem*{introtheorem}{Theorem}
\newtheorem{lemma}[theorem]{Lemma}
\newtheorem{proposition}[theorem]{Proposition}
\newtheorem{corollary}[theorem]{Corollary}
\theoremstyle{remark}
\newtheorem{observation}[theorem]{Remark}

\newtheorem{definition}{Definition}[section]

\newcommand{\les}{\lesssim}
\newcommand{\dd}{\,d}
\newcommand{\R}{\mathbb R}
\newcommand{\C}{\mathbb C}
\newcommand{\Z}{\mathbb Z}
\newcommand{\lb}{\label}
\newcommand{\be}{\begin{equation}}
\newcommand{\ee}{\end{equation}}
\newcommand{\B}{\mathcal B}
\newcommand{\mc}{\mathcal}
\newcommand{\K}{\mc K}
\newcommand{\W}{\mc W}

\DeclareMathOperator{\supp}{supp}
\DeclareMathOperator{\BMO}{BMO}
\DeclareMathOperator{\p}{\mathbb P}
\renewcommand{\P}{\p}
\DeclareMathOperator{\E}{\mathbb E}

\DeclareMathOperator{\rere}{Re}
\renewcommand{\Re}{\rere}

\title[Spectral multipliers and wave propagation]{Spectral multipliers and Wave Propagation for Hamiltonians with a scalar potential}
\author{Marius Beceanu}
\address{University at Albany SUNY, Mathematics and Statistics Department, 1400 Washington Ave., Albany, NY 12222}
\email{mbeceanu@albany.edu}
\author{Michael Goldberg}
\address{University of Cincinnati, Department of Mathematical Sciences, 2815 Commons Way, Cincinnati, OH 45221-0025}
\email{goldbeml@ucmail.uc.edu}

\begin{document}

\begin{abstract} We extend several fundamental estimates regarding spectral multipliers for the free Laplacian on $\R^3$ to the case of perturbed Hamiltonians of the form $-\Delta+V$, where $V$ is a scalar real-valued potential.

In this paper, we prove resolvent estimates, a dispersive bound for the perturbed wave propagator, Mihlin multiplier and fractional integration bounds, and the full range of wave equation Strichartz estimates, under optimal or almost optimal scaling-invariant conditions on the potential and on the spectral multipliers themselves.
\end{abstract}
\maketitle

\tableofcontents
\section{Introduction}

In this paper we establish three properties of the spectral measure of three-dimensional Schr\"odinger operators $H = -\Delta + V$ with potentials belonging to a scaling-critical class.   The first set of results concerns pointwise behavior of the resolvents
$R_V(z) = (H - z)^{-1}$ as integral operators on $\R^3$, in particular controlling the limits as $z$ approaches the positive real axis.  One of the key estimates is equivalent to a statement about the ``sine" propagator of the linear wave equation, and it will be proved in that context.  Some useful bounds on the exponential decay of eigenfunctions are obtained as a consequence.

The second set of results concerns the $L^p$-boundedness of spectral multipliers $m(H)$.  We prove that if $m$ satisfies the conditions of the classical H\"ormander-Mihlin theorem, then $m(H)$ is a bounded operator on $L^p(\R^3)$, $1 < p < \infty$.  While the integral kernel of $m(H)$ strongly resembles that of a Calder\'on-Zygmund singular integral, we leave the analysis of its $L^1$ and Hardy space behavior for a future paper.  With a slightly stronger, but still scale-invariant condition on the potential, we are able to show that $m(H)$ is weak-type (1,1) using a different argument.  Sobolev inequalities for factional integration are recovered as well.

Finally, we prove Strichartz estimates for the perturbed wave equation provided the potential also belongs to the Lorentz space $L^{3/2, \infty}$, following an argument due to Beals~\cite{beals}.

\subsection{Setup}
Consider Hamiltonians of the form $H=-\Delta+V$ in $\R^3$, i.e.\;in three space dimensions. Here the Laplacian $-\Delta$ is the free Hamiltonian, corresponding to free motion, while $V$ is a real-valued scalar potential. 

Our basic assumption is that $V$ belongs to the closure of $C^\infty_c(\R^3)$ with respect to the global Kato norm
\be\lb{katonorm}
\|V\|_\K := \sup_{y \in \R^3} \int_{\R^3} \frac{|V(x)| \dd x}{|x-y|}
\ee
and we denote this space by $\K_0$. Then $H=-\Delta+V$ is self-adjoint and bounded from below~\cite{simon}.  Multiplication by $V$ is a compact perturbation relative to the Laplacian, thus the essential spectrum of $H$ remains $[0,\infty)$ and there can be at most finitely many eigenvalues away from $0$.  If one assumes that $V$ is in $ L^{3/2,1}(\R^3) \subset \K_0$ (a weaker condition suffices) is real-valued, then there are no embedded eigenvalues or resonances in the continuous spectrum ($\lambda>0$), see \cite{ioje} and \cite{kota}. With our baseline assumption $V \in \K_0$, there are no embedded resonances, see~\cite{goldberg}. Where necessary, we assert the absence of embedded eigenvalues as an independent assumption.

For some results in this paper we need to further assume that $V \in \K_0 \cap L^{3/2,\infty}(\R^3)$. Note that both the global Kato norm and the Lorentz norm $L^{3/2, \infty}(\R^3)$ scale similarly to $|x|^{-2}$, which is critical with respect to the Laplacian.  One result (Theorem~\ref{thm:weakL1}) requires a modified Kato norm with $\ell^1$ summability across dyadic length scales. The relevant norm definition will be introduced in the statement of that theorem.

Both our pointwise bounds for the wave equation and the Strichartz inequalities apply to the absolutely continuous part of the propagation, that is we first apply an orthogonal projection away from the space of bound states.  Bound states at the threshold of the continuous spectrum (i.e.\; at zero energy) can be present even for $V \in C^\infty_c$ and in general destroy dispersive estimates, even after first projecting the bound states away.  Throughout this paper we make an \emph{a priori} assumption that there are no obstructions at zero energy.

The free Hamiltonian is a positive operator, thus one common way to prevent bound states is to require the potential to be nonnegative as well.  When this is the case, one can often also weaken the regularity and decay assumptions on $V$ substantially, for example letting $V \geq 0$ belong to a reverse H\"older class.  Bound states can also be avoided if the negative part of the potential, $V_- := \max(-V, 0)$, is present but sufficiently small.  The sharp smallness condition $\| V_- \|_{\K} < 4\pi$ arises naturally and has been exploited by~\cite{dancpier}.

Propagators for the perturbed wave equation and Mihlin multipliers are both examples of spectral multipliers. 
One can use functional calculus to define a family of commuting operators $f(H)$, called spectral multipliers, at first for analytic functions $f$ by a Cauchy integral around the spectrum.  Since $H$ is self-adjoint here, $f(H)$ can be defined in terms of the spectral measure
\be \lb{functionalcalc}
f(H) = \int_{\sigma(H)} f(\lambda) \dd E_H(\lambda)
\ee
which exists as a bounded operator on $L^2$ whenever $f$ is a bounded Borel measurable function on $\sigma(H)$ (the spectrum of $H$).

Commonly studied classes of multipliers include the resolvent $R_V(\lambda)=(H-\lambda)^{-1}$, fractional integration operators $H^{-s}$, $s>0$, Mihlin multipliers $m(H)$, where $m$ is a symbol such that $m^{(n)}(\lambda) \les \lambda^{-n}$, Bochner-Riesz sums, the Schr\"{o}dinger evolution $e^{itH}$, the wave and half-wave propagators $\cos(t\sqrt H)$, $\frac {\sin(t\sqrt H)}{\sqrt H}$, and $e^{it\sqrt H}$, the heat flow $e^{-tH}$, and others.

In the free case, where $H = -\Delta$, each spectral multiplier $f(-\Delta)$ corresponds to the Fourier multiplier $f(|\xi|^2)$.  
The following properties are well-known:

1. Resolvent bounds: The free resolvent $R_0(z) := (-\Delta-z)^{-1}$ is analytic on $\C \setminus [0, \infty)$ (where $\sigma(-\Delta)=[0, \infty)$), having the explicit kernel
\be\lb{explicit}
R_0(z)(x, y) = \frac 1 {4\pi} \frac {e^{-\sqrt{-z}}|x-y|}{|x-y|}. 
\ee
Here $\sqrt z$ is the main branch of the square root.

Although $R_0(z)$ does not exist as a bounded operator on $L^2$ for $z \in [0,\infty)$, it has a continuous (in e.g.\ the strong $\B(L^{6/5, 2}, L^{6, 2})$ topology, not in the norm topology) extension to the positive real halfline, where it takes different boundary values on each side for $\lambda>0$:
$$
R_0^\pm(\lambda)(x,y) := R_0(\lambda \pm i0)(x, y) = \frac 1 {4\pi} \frac {e^{\pm i\sqrt\lambda |x-y|}}{|x-y|}.
$$
The existence of such extensions in a weaker topology than $\B(L^2)$ is known as the limiting absorption principle.
It is clear from the formulas that the free resolvent obeys pointwise bounds uniformly in $\lambda \in \C$,
\be \lb{uniffreeresolvent}
|R_0(\lambda)(x, y)| \leq \frac 1 {4\pi |x-y|}.
\ee

2. Fractional powers of the Laplacian: For $0<\alpha<3$, $(-\Delta)^{-\alpha/2}$ has the explicit kernel
$$
(-\Delta)^{-\alpha/2}(x, y) = C_\alpha |x-y|^{\alpha-3},
$$
and it is bounded as an operator from $L^p(\R^3)$ to $L^q(\R^3)$ for $\frac{3}{p} - \frac{3}{q} = \alpha$, $1<p, q < \infty$ by the Hardy-Littlewood-Sobolev inequality.

3. Schr\"{o}dinger and wave propagators: the free Schr\"{o}dinger and wave propagators have the kernels
$$
e^{it\Delta}(x, y) := \frac 1 {(4\pi i t)^{3/2}} e^{-\frac{|x-y|^2}{4it}},\ \frac {\sin(t\sqrt{-\Delta})}{\sqrt{-\Delta}}(x, y) := \frac 1 {4\pi t} \delta_{|t|}(|x-y|).
$$
Hence one can obtain several dispersive estimates, such as pointwise decay, Strichartz, and reversed Strichartz estimates (see \cite{becgolschr} and \cite{becgol} for more details).  The bound of greatest interest in this paper is as follows:
\be \lb{integratedwave}
\int_{-\infty}^\infty \Big|\frac{\sin(t\sqrt{-\Delta})}{\sqrt{-\Delta}}(x,y)\Big|\,dt = \frac{1}{2\pi|x-y|},
\ee
where we are being imprecise in notation with regard to the integration of delta-functions.  Delta functions are also present in the key integral estimates~\eqref{eq:goal} and~\eqref{eq:unifWiener} below.

4. Mihlin spectral multipliers: In the free case, Mihlin multipliers $m(H)$ are bounded on $L^p$, $1<p<\infty$, and of weak $(1, 1)$ type, i.e.\ bounded from $L^1$ to $L^{1, \infty}$.

There are several possible choices for the definition of Mihlin symbols $m$ that lead to substantially similar properties. The classical condition is that $m^{(n)}(\lambda) \leq \lambda^{-n}$ for $n \geq 0$. One can weaken this condition to
\be \lb{Mihlin}
\sup_{\beta > 0} \|\phi(\lambda) m(\beta \lambda)\|_{\dot H^{s}} < \infty
\ee
for a given smooth function $\phi \in C^\infty_c((\frac12,2))$ and real number $s > \frac32$ and still obtain the same results.

Imaginary powers of the Laplacian, given by $m(-\Delta)=(-\Delta)^{i\sigma}$, $\sigma \in \R$, are a particularly useful family of Mihlin multipliers, with the norm in~\eqref{Mihlin} growing at the rate $|\sigma|^s$.  They are often called upon to provide boundary values with complex interpolation, see for example Sections~\ref{sec:Fractional} and~\ref{sec:Strichartz} below. 

All these spectral multipliers are $L^2$-bounded, except for the fractional powers of the Laplacian $(-\Delta)^{-s}$ and the boundary values of the resolvent along the real line.

\subsection{Summary of Results} Our goal is to prove as many of the above estimates as possible for the perturbed Laplacian $H=-\Delta+V$.  Because our assumptions permit $H$ to have a finite number of strictly negative eigenvalues $-\mu_j^2$, the wave equation admits solutions of the form $\frac{\sin(t\sqrt{H})}{\sqrt{H}} f_j = \frac{\sinh(\mu_j t)}{\mu_j} f_j$ for the associated eigenfunction $f_j$.  [The fact that $z^{-1/2}\sin(t\sqrt{z})$ is analytic on $\R$ allows the multiplier to be well defined even though $H$ is not positive.]  These solutions clearly violate~\eqref{integratedwave}.  In order to prevent these special solutions from dominating all other aspects of the wave propagator, we first apply the projection $P_c$ onto the continuous spectrum of $H$.  Our main result in this direction is as follows:

\begin{introtheorem}[Theorem~\ref{thm:main}]
 Assume $V \in \K_0$, and $H = -\Delta + V$ has no eigenvalue or resonance at
zero, and no positive eigenvalues.  Then
\begin{equation} \label{integratedperturbedwave}
\int_{-\infty}^\infty 
\Bigl|\frac{\sin(t\sqrt{H})}{\sqrt{H}}P_c (x,y)\Bigr|\,dt \les \frac{1}{|x-y|}.
\end{equation}
\end{introtheorem}

The wave equation propagator is very closely tied to the Fourier transform of the operator-valued function $\lambda \to R_V^+(\lambda^2) := R_V((\lambda + i0)^2)$, where $R_V(z) = (-\Delta + V -z)^{-1}$ is the resolvent of $H$.  Note that for negative values of $\lambda$, $R_V^+((-\lambda)^2) = R_V^-(\lambda^2)$ according to this definition.  Theorem~\ref{thm:main} almost immediately implies a pointwise bound for resolvent kernels continued to the positive real axis:
\begin{introtheorem}[Corollary~\ref{cor:resolventkernel}]
Under the same conditions as Theorem~\ref{thm:main},
$$
|R_V^\pm(\lambda)(x, y)| \les \frac{1}{|x-y|} \quad \text{uniformly in } \lambda \geq 0.
$$
\end{introtheorem}

We also show that the perturbed wave equation still has finite speed of propagation.  Thus the portion of $\frac{\sin(t\sqrt{H})}{\sqrt{H}}P_c$ lying in the region $\{|x-y| > |t|\}$, i.e.\;exterior to the light cone, is created entirely by the projection $P_c$ and is a linear combination of the hyperbolic sine solutions.  In combination with~\eqref{integratedperturbedwave} this gives exponential bounds for the projection onto each eigenspace of bound states.

The proof of Theorem~\ref{thm:main} is the longest section of the paper.  It relies on the operator-valued Wiener inversion 
methods originating in~\cite{bec}, and specifically uses a version stated in~\cite{becgol}.  This argument is brought to bear on operators $R_0^+((\,\cdot\,)^2)V$ acting on weighted spaces $\frac{1}{|x-y_0|}L^\infty(\R^3)$.  We need to show that not only does this application fit within the established framework of operator-valued inversion lemmas (after some algebraic manipulation), but there is also uniformity with respect to the choice of $y_0 \in \R^3$.  

Resolvent operators are an essential ingredient in the study of spectral multipliers, as the Stone formula for the continuous spectral measure $dE_H(\lambda)$ of a self-adjoint operator (as in~\eqref{functionalcalc}) is precisely
$$
dE_H(\lambda) = (2\pi i)^{-1}[(H- (\lambda +i0))^{-1} - (H-(\lambda -i0))^{-1}],
$$
or in our case with $H = -\Delta + V$, $dE_H(\lambda) = (2\pi i)^{-1}[R_V^+(\lambda) - R_V^-(\lambda)]$.  The bound~\eqref{integratedperturbedwave} provides us with the entry point for proving a H\"ormander-Mihlin theorem for Schr\"odinger operators.

There are several such theorems currently in the literature, each one having different assumptions about the operator $H$ or the multiplier function $m$.   If $H \geq 0$ (for example if $V \geq 0$) then it fits into the broad framework of spectral calculus articulated in~\cite{sikora}, where many $L^p$ bounds for radial Fourier multipliers can be transfered over to their perturbed counterpart.  If it is known that the wave operators of $H$ are bounded on all $L^p(\R^3)$ (as occurs with potentials of the type considered in~\cite{bec2} or~\cite{becschlag}), then $L^p$ bounds for spectral multipliers $m(H)$ are once again automatically inherited from the properties of $m(-\Delta)$.  This can hold even if $H$ admits bound states with negative energy.  Results in~\cite{zheng} and~\cite{hong} also hold for Schr\"odinger operators with bound states and rough potentials; the former imposes auxiliary conditions on frequency-localized multipliers of $H$ which are difficult to check, and the latter requires stronger regularity $s > 2$ in~\eqref{Mihlin}.

We prove the following multiplier bounds.

\begin{introtheorem}[Theorem~\ref{thm:multiplier3}]

 Assume $V \in \mc K_0 \cap L^{3/2, \infty}$, and $H = -\Delta + V$ has no eigenvalue or resonance at
zero, and no positive eigenvalues.  Let $\phi$ be a $C^\infty$ function supported on
$[\frac12, 2]$ such that $\{\phi(2^{-k}\lambda)\}_{k\in\mathbb Z}$ forms a partition of unity for
$\R^+$.

Suppose $m: (0,\infty) \to \mathbb C$ satisfies
$\sup_{k\in \Z} \|\phi(\lambda)m(2^k\lambda)\|_{H^s} = M_s < \infty$ for some $s> \frac32.$
Then $m(\sqrt{H}) := m(\sqrt{H P_c})$ is a singular integral operator that is bounded on
$L^p(\R^3)$ for all $1 < p < \infty$, with operator norm less than $C_{s,p} M_s$.

If $V \in \mc K_0$ and the bound above holds for some $s>2$, then $m(\sqrt{H})$ has an integral kernel 
$K(x,y)$ that is bounded pointwise by $|K(x,y)| \les \frac{M_s}{|x-y|^3}$.
\end{introtheorem}

\begin{observation}
The expression $m(\sqrt{H})$ is chosen here to be comparable to the Fourier multiplier $m(|\xi|)$.
The conclusions of Theorem~\ref{thm:multiplier3} hold equally well for $m(H)$, via the reduction 
$\|\phi(\lambda) m(2^{2k} \lambda^2)\|_{H^s} \sim \|\phi(\sqrt{\lambda})m(2^{2k}\lambda)\|_{H^s}
= \|\phi(\sqrt{\lambda})(\sum_{j=-2}^2 \phi(2^j\lambda)m(2^{2k}\lambda))\|_{H^s} \les M_s$.
The first equivalence comes from a smooth change of variables $\lambda \to \sqrt{\lambda}$ on the
support of $\phi$, and the rest are due to partition of unity properties of $\phi$ and linear changes of variable.
\end{observation}

Our proof of Theorem~\ref{thm:multiplier3} roughly follows the plan of showing that $m(\sqrt{H})$ is a Calder\'on-Zygmund operator.  However it is not clear whether its kernel has the standard cancellation properties.  

This has two consequences.  The first is that homogenoeus pointwise bounds on the kernel do not, by themselves, imply boundedness of the operator.  As a result we eventually pursue a weak-$(p,p)$ bound on the difference $m(\sqrt{H}) - m(\sqrt{-\Delta})$ for $1 < p \leq 2$ based on more direct $L^p_x L^{p',\infty}_y$ estimates of its kernel. One technical step requires the potential $V(x)$ to belong to $L^{3/2,\infty}(\R^3)$.
The second consequence is that the argument does not extend to the weak-$(1,1)$ endpoint. We establish the weak-$(1,1)$ bound under a slightly more restrictive set of hypotheses than what is required for the pointwise kernel bound.
\begin{introtheorem}[Theorem~\ref{thm:weakL1}]
Suppose $V \in \mc K_0$ and the conditions of Theorem~\ref{thm:multiplier3} are otherwise satisfied for some $s>2$. If the potential $V(x)$ additionally satisfies
$$
\|V\|_{\tilde \K} := \sup_{x \in \R^3} \sum_{\ell \in \mathbb Z} \|\chi(|z-x| \sim 2^\ell) V(z)\|_{\K_z} < \infty,
$$
then $m(\sqrt{H})$ maps $L^1(\R^3)$ to $L^{1,\infty}(\R^3)$ and maps $L^p(\R^3)$ to itself for all $1 < p < \infty$.
\end{introtheorem}
Note that this condition, which is slightly stronger than the Kato class condition, is not vacuous. In particular, any $\dot W^{1, 1}$ potential satisfies it, by Lemma \ref{w11embed}.

This time the proof is more directly perturbative from the free case.  We decompose the difference $m(\sqrt{H}) - m(\sqrt{-\Delta})$ into two terms, one of which is a Calder\'on-Zygmund operator and the other is dominated pointwise by the Hardy-Littlewood maximal function.

The techniques from Theorem~\ref{thm:multiplier3} are also sufficient to prove pointwise and $L^p \to L^q$ bounds for a class of spectral multipliers that includes the fractional powers $H^{-\alpha/2}P_c$, $0 < \alpha < 3$.  More precisely we prove:
\begin{introtheorem}[Theorem \ref{thm:multiplier4}]
Let $\alpha \in (0,3)$ and suppose the conditions of Theorem~\ref{thm:multiplier3} are satisfied for some $s> \max(\frac32 - \alpha, 1 - \frac23 \alpha)$.  Then $H^{-\alpha/2}m(\sqrt{H})$ maps $L^p(\R^3)$ to $L^q(\R^3)$ for all pairs $1 < p,q < \infty$ with
$\frac{3}{p} - \frac{3}{q} = \alpha$.

If $V \in \mc K_0$ and the conditions are otherwise satisfied for some $s > 2-\alpha$, then the integral kernel of $H^{-\alpha/2}m(\sqrt{H})$ is bounded pointwise by $|x-y|^{\alpha - 3}$.
\end{introtheorem}

The concluding discussion of Strichartz estimates 
proves the full range of Strichartz estimates for the perturbed wave equation $u_{tt} + Hu = F(t,x)$.
The minimal regularity of the potential $V \in \mc K_0 \cap L^{3/2,\infty}(\R^3)$ is enough to insure that Sobolev spaces $H^{s}(\R^3)$, $|s| < \frac32$ are equivalent whether they are defined relative to the the free or perturbed Laplacian, and similarly for a range of spaces $W^{s,p}(\R^3)$ that arise in the course of the proof.

\subsection{Comparison with previous results}

Previous work on spectral multipliers for $-\Delta+V$ was done by Zheng \cite{zheng}, Hong \cite{hong}, and Sikora--Yan--Yao \cite{sikora}, as well as in \cite{becgol}.

Zheng \cite{zheng} proved the $L^p$ and weak $L^1$ boundedness of Mihlin multipliers in every dimension, for potentials of arbitrary sign, using a certain pair of conditions on the Hamiltonian $-\Delta+V$ itself. It is not clear to what condition on the potential $V$ this corresponds, except in dimension one.

Sikora--Yan--Yao \cite{sikora} used heat semigroups and Gaussian kernel estimates to prove both $L^p$ estimates for Mihlin multipliers and resolvent estimates. Their results apply to some extent in a much wider setting, including the bi-Laplacian and on compact manifolds. However, they did not look for sharp conditions on the potential and, due to the method, the results only apply when the Hamiltonian is positive semi-definite.

Hong \cite{hong} proved the $L^p$ boundedness of Mihlin multipliers in three dimensions, for $-\Delta+V$ where $V \in \mc K_0 \cap L^{3/2, \infty}$ and where the Mihlin multiplier $m$ is locally in $H^s$, $s>2$.

Our current results also hold in three dimensions, though they have parallels in dimensions one and two as well. We generalize the results of Hong \cite{hong} in two ways: under the same conditions, we replace the exponent $s>2$ with $s>3/2$, same as the optimal range in the classical H\"{o}rmander theorem for $-\Delta$, and under a slightly stronger condition on $V$ we also prove the weak-$L^1$ boundedness of Mihlin multipliers when $s>2$.

We also include several results with no counterparts in the above-mentioned papers, dealing with other classes of spectral multipliers, from a completely new estimate for the resolvent to bounds for negative powers of the Hamiltonian (fractional integration) to Strichartz estimates for the wave propagators.

\subsection{Notations}
We use the following standard notations for the Fourier transform and its inverse:
$$
[\mc F f](\xi) = \int_{\R^d} f(x) e^{-ix\xi} \dd x
$$
and
$$
[\mc F^{-1} g](x) = \frac 1 {(2\pi)^d} \int_{\R^d} g(\xi) e^{ix\xi} \dd \xi.
$$

\section{Pointwise bound for the wave propagator $\frac{\sin(t\sqrt{H})}{\sqrt{H}}$}
Here we prove the following estimate for the perturbed wave
equation:
\begin{theorem} \label{thm:main}
 Assume $V \in \K_0$, and $H = -\Delta + V$ has no eigenvalue or resonance at
zero, and no positive eigenvalues.  Then
\begin{equation}
\int_{-\infty}^\infty 
\Bigl|\frac{\sin(t\sqrt{H})}{\sqrt{H}}P_c (x,y)\Bigr|\,dt \les |x-y|^{-1}.
\end{equation}
\end{theorem}

\begin{observation}
There may be a shorter proof that combines the time-weighted bound for the sine propagator in~\cite[Theorem 1]{becgol} with support properties derived as Proposition~\ref{prop:cone} in the next section and subsequent analysis of eigenspace projections.  The argument here is longer but self-contained modulo proving the Wiener inversion theorem (Theorem~\ref{thm:inversion}), and it develops the framework needed to prove assertions in Section~\ref{sec:lightcone}.
\end{observation}

\begin{proof}
First use the Stone formula for spectral measure to re-write
\begin{align}
H^{-1/2}\sin(t\sqrt{H})P_c &= 
(2\pi i)^{-1}\int_0^\infty \lambda^{-1/2}\sin(t\sqrt{\lambda}) 
(R_V^+(\lambda) - R_V^-(\lambda))\,d\lambda \nonumber \\
&= (\pi i)^{-1} \int_0^\infty \sin(t\lambda)(R_V^+(\lambda^2)-R_V^-(\lambda^2))
\, d\lambda \nonumber\\
&= (\pi i)^{-1} \int_{-\infty}^\infty \sin(t\lambda)R_V^+(\lambda^2)\,d\lambda.
\label{eq:stone}
\end{align}
Noting that $\sin(t\lambda) = (2i)^{-1}(e^{i\lambda} - e^{-i\lambda})$, 
it suffices to prove that
\begin{equation} \label{eq:goal}
\int_{-\infty}^\infty 
|\mathcal{F}(R_V^+({\scriptstyle (\,\cdot\,)^2})) (t,x,y)|\, dt \les |x-y|^{-1}
\end{equation}
where $t$ is taken to be the Fourier dual variable to $\lambda^2$.  
This is of course true for the free resolvent 
$R_0^+(\lambda^2)(x,y) = (4\pi|x-y|)^{-1}e^{i\lambda|x-y|}$.  The issue is to
preserve inequality~\eqref{eq:goal} through the course of the resolvent 
identity
\begin{equation*}
R_V^+(\lambda^2) = (I + R_0^+(\lambda^2)V)^{-1}R_0^+(\lambda^2).
\end{equation*}
The key estimate we will need is stated below.

\begin{lemma} \label{lem:goal}
Assume $V \in \mathcal{K}_0$, and there is no eigenvalue or resonance at zero
and no positive eigenvalues.
There exists $C < \infty$ so that for every $x,y \in \R^3$,
\begin{equation} \label{eq:unifWiener}
\int_{\R^3} \int_{-\infty}^\infty
|\mathcal{F}(I + R_0^+({\scriptstyle (\,\cdot\,)^2})V)^{-1}|(t,x,z)
\frac{|x-y|}{|z-y|}\,dt\, dz< C
\end{equation}
\end{lemma}

The remaining steps are as follows.
Taking the Fourier transform of the resolvent identity yields
\begin{equation*}
\mathcal{F} R_V^+(t,x,y) = \int_{\R^3}\int_{-\infty}^\infty
\mathcal{F}(I + R_0^+V)^{-1}(t-s,x,z) \mathcal{F} R_0^+(s,z,y)\,ds\,dz
\end{equation*}
where we have suppressed the dependence on $\lambda^2$.
The integral of this expression with respect to $t$ is then bounded by
\begin{multline*}
\frac{1}{|x-y|} \Big(\int_{\R^3}\int_{-\infty}^\infty 
|\mathcal{F}(I+R_0^+V)^{-1}|(t-s,x,z)\frac{|x-y|}{|z-y|}\,dt\,dz \Big) \\
\Big(\sup_z |z-y| \int_{-\infty}^\infty |\mathcal{F}R_0^+|(s,z,y)\,ds\Big).
\end{multline*}
Thanks to the explicit formula for the free resolvent, the last expression
in parentheses is exactly $\frac{1}{4\pi}$.  The first expression 
is bounded by a constant in~\eqref{eq:unifWiener} which leads to
an overall bound of $\frac{C}{4\pi|x-y|}$.
\end{proof}

\begin{proof}[Proof of Lemma~\ref{lem:goal}] 
Verification of the bound~\eqref{eq:unifWiener} proceeds in two steps.  First we show
that the inequality is satisfied uniformly in $x$ for each fixed $y \in \R^3$,
with an associated constant $C(y)$.  Then we demonstrate uniform boundedness
of $C(y)$.  While the two arguments employ a very similar mechanism, there
does not appear to be a way to accomplish both goals at the same time.
The first step will dominate the exposition as relevant methods and
technical hypotheses are introduced.

Given a point $y \in \R^3$, let $M_y$ denote pointwise multiplication by the
function $|\,\cdot\, - y|$.  The action of $M_y$ commutes with the Fourier
transform in the $t$ and $\lambda$ variables. Then~\eqref{eq:unifWiener}
can be re-written as
\begin{equation} \label{eq:unifWiener2}
\int_{\R^3} \int_{-\infty}^\infty
|\mathcal{F}(I + M_y R_0^+({\scriptstyle (\,\cdot\,)^2})V M_y^{-1})^{-1}|(t,x,z)
\,dt\, dz< C.
\end{equation}

Operator inverse estimates of this type can be approached by an
appropriate generalization of the Wiener algebra $\mc{F}(L^1(\R))$.  Let $X$
be a Banach lattice (in this case over the measure space $(\R^3, dx)$).

\begin{definition}
We say a family of operators $T(\lambda)$ with integral kernel 
$T(\lambda,x,z)$, $\lambda \in \R$, belongs 
to the space $\W_X$ if its Fourier transform $\mc{F}T(t,x,z)$ is a measure satisfying
\begin{equation} \label{eq:W_X}
\Big\| \int_{-\infty}^{\infty} |\mc{F}T(t,x,z)|\,dt \Big\|_{\B(X)} < \infty.
\end{equation}
\end{definition}
The quantity in~\eqref{eq:W_X} will be taken as the norm $\| T\|_{\W_X}$,
and it is straightforward to show that $\W_X$ is an algebra using
Fubini's theorem and the lattice properties of $X$.

In this language, \eqref{eq:unifWiener2} asserts that
$(I + M_yR_0(\lambda^2)VM_y^{-1})^{-1}$
belongs to $\W_{L^\infty}$, uniformly in $y\in\R^3$.

Observe that $M_yR_0^+(\lambda^2)VM_y^{-1}$ belongs to
$\W_{L^\infty}$ for each $y \in \R^3$.
To see this, note that
its Fourier transform has the kernel  $\frac{|x-y| V(z)\delta_0(t-|x-z|)}{4\pi|x-z||z-y|}$.
When integrated over $t$, the result has the property
\begin{equation} \label{eq:computation}
\begin{aligned}
\int_{\R^3}\int_{-\infty}^\infty &\frac{|x-y| |V(z)| \delta_0(t-|x-z|)}{4\pi |x-z|\,|z-y|}\,dt\,dz \\
&= \int_{\R^3} \frac{|x-y| |V(z)|}{4\pi |x-z|\,|z-y|}\,dz \\
&= \int_{\R^3} \frac{|x-y|}{4\pi(|x-z|+|z-y|)}\Big(\frac{|V(z)|}{|x-z|} + \frac{|V(z)|}{|z-y|}\Big)\,dz \\
&\leq (2\pi)^{-1}\|V\|_{\K} \quad \text{ for every } x \text{ and } y.
\end{aligned}
\end{equation}

The constant operator $T(\lambda) = I$ is the identity element of $\W_X$
for any $X$, so it certainly belongs to $\W_{L^\infty}$ as well.
The remaining questions are whether $I + M_yR_0^+(\lambda^2)VM_y^{-1}$
is an invertible element of $\W_{L^\infty}$ for each $y$, and whether
there is a uniform bound on the size of its inverse.

We borrow a result concerning existence of inverses in $\W_X$ from
\cite{becgol}.  To state it succinctly, define the ``high-frequency cutoff"
$S_R \in \B(\W_X)$ for each $R>0$ by the property
\begin{equation}
\mc{F}(S_R T)(t,x,z) = \mc{F} T(t,x,z)\chi_{|t| > R}.
\end{equation}

\begin{theorem}[\cite{becgol}, Proposition 10] \label{thm:inversion}
Suppose $T \in \W_X$ is such that
\renewcommand{\labelenumi}{(C\theenumi)}
\begin{enumerate}
\item $\lim\limits_{R\to\infty} \|S_R T\|_{\W_X} = 0$.  \label{C1}
\item For some $N > 0$, $\lim\limits_{\delta \to 0} 
 \|(e^{i\delta\lambda}-1)T^N(\lambda)\|_{\W_X} = 0$. \label{C2}
\end{enumerate}
If $I + T(\lambda)$ is an invertible element of $\B(X)$ 
for every $\lambda \in \R$, then $I + T$ is invertible in $\W_X$.
\end{theorem}

Unfortunately, $M_yR_0^+(\lambda^2)VM_y^{-1}$ does not satisfy 
condition~(C\ref{C1}) in our Inversion Theorem.  More specifically
\begin{equation*}
\|S_R(M_yR_0^+VM_y^{-1})\|_{\W_{L^\infty}} \geq
\liminf_{|x| \to \infty} \int_{|x-z| > R} \frac{|x-y|V(z)}{4\pi |x-z|\,|z-y|}\,dz 
\geq \int_{\R^3} \frac{|V(z)|}{4\pi |z-y|}\,dz
\end{equation*}
by Fatou's lemma.

For this reason we expand the inverse using the identity
\begin{equation} \label{eq:switch}
(I + M_yR_0^+VM_y^{-1})^{-1} = I + M_yR_0^+M_y^{-1}
(I + M_yVR_0^+M_y^{-1})^{-1} V.
\end{equation}
We will apply Theorem~\ref{thm:inversion} to show that
$(I + M_yVR_0^+M_y^{-1})^{-1} \in \W_{\K}$.
Meanwhile $V$ clearly maps from $L^\infty$ to $\K$.  Another calculation nearly
identical to~\eqref{eq:computation} shows that 
$M_y R_0^+(\lambda^2)M_y^{-1}$  maps $\W_\K$ to $\W_{L^\infty}$
because
\begin{multline*}
\sup_x \Big\| \int_{-\infty}^\infty |\mc{F}(M_yR_0^+({\scriptstyle (\,\cdot\,)^2})M_y^{-1}|(t,x,z)
\,dt \Big\|_{\K^*_z} 
= \sup_x \Big\|\frac{|x-y|}{4\pi |x-z|\,|z-y|}\Big\|_{\K^*_z} \\
\leq \sup_x \Big\| \frac{|x-y|}{4\pi(|x-\,\cdot\,| + |\,\cdot\,-y|)}\Big\|_{\infty} 
\Big\|\Big(\frac{1}{|x-\,\cdot\,|} +\frac{1}{|\,\cdot\,-y|}\Big)\Big\|_{\K^*} \leq (2\pi)^{-1}.
\end{multline*}

The operator $M_yVR_0^+(\lambda^2)M_y^{-1}$
belongs to $\W_\K$ for each $y$, with norm bounded by
$(2\pi)^{-1} \|V\|_\K$.  
We will show that it meets all criteria of Theorem~\ref{thm:inversion}.

We present the main conditions of the Inversion Theorem on $\W_\K$
that need to be checked as a lemma to be proven later.
\begin{lemma} \label{lem:conditions}
If $V \in \K_0$, then for each $y \in \R^3$,
\begin{equation} \label{eq:Kstarbound}
\Big\| \frac{\chi_{\{|x-\,\cdot\,| > R\}}}{|x-\,\cdot\,|\,|\,\cdot\,-y|}\Big\|_{\K^*}
\leq 2 \max(|x-y|,R)^{-1}
\end{equation}
and
\begin{equation} \label{eq:Kbound}
\lim_{R \to \infty} \Big\| \frac{ |\,\cdot\, - y| V( \,\cdot\,)}{\max(|\,\cdot\, - y|,R)}\Big\|_{\K} = 0.
\end{equation} 
Furthermore, if $V$ is bounded with compact support then there exists $N < \infty$ such that 
\begin{equation} \label{eq:highenergy}
\| M_y (R_0^+(\lambda^2)V)^{N-1} M_y^{-1}\|_{\B(L^\infty)} \les (1+ |\lambda|)^{-3}.
\end{equation}
Finally, if $V \in \K_0$ is real valued and $-\Delta + V$ has no resonance at zero and
no nonnegative eigenvalues, then $I+ M_yVR_0^+(\lambda^2)M_y^{-1}$ is invertible in $\B(\K)$
for all $y \in \R^3$ and $\lambda \in \R$. 
\end{lemma}

Assuming Lemma~\ref{lem:conditions}, let us verify that $M_y V R_0^+(\lambda^2)M_y^{-1}$ 
satisfies Condition~(C\ref{C1}).
A direct calculation shows that its Fourier transform has integral kernel
$$
K_y(t,x,z) = \frac{|x-y|V(x) \delta_0(t-|x-z|)}{4\pi |x-z| |z-y|}.
$$
Applying~\eqref{eq:Kstarbound} and~\eqref{eq:Kbound} in turn demonstrates that
the norm of $\int_\R| K_y|\chi_{\{|t|>R\}}\,dt$ as an element of $\B(\K)$
converges to zero as $R$ increases to infinity.

To verify that $M_y V R_0^+(\lambda^2)M_y^{-1}$ satisfies 
Condition~(C\ref{C2}), we imitate an argument from Section~4
of~\cite{becgolschr}.  Choose a bounded compactly supported approximation
$V_\epsilon$ with $\|V-V_\epsilon\|_{\K} < \epsilon$ and 
$\|V_\epsilon\|_{\K} \leq \|V\|_{\K}$.  Then choose $R$ large enough
so that
\begin{equation*}
\big\| S_R(M_yV_\epsilon R_0^+({\scriptstyle (\,\cdot\,)^2})M_y^{-1})\|_{\W_\K} < 
\epsilon\big({\textstyle \frac{\|V\|_{\K}}{\pi}}\big)^{1-N}.
\end{equation*}
It follows that for $\tilde{R} \geq NR$,
\begin{multline} \label{eq:S_NR}
\big\|  S_{\tilde{R}}(M_yV_\epsilon R_0^+({\scriptstyle (\,\cdot\,)^2})M_y^{-1})^N\|_{\W_\K}
\\
\begin{aligned}&\leq \big\| \big(M_yV_\epsilon R_0^+({\scriptstyle (\,\cdot\,)^2})M_y^{-1}\big)^N
- \big((I - S_R)(M_yV_\epsilon R_0^+({\scriptstyle (\,\cdot\,)^2})M_y^{-1})\big)^N
\big\|_{\W_{\K}}
\\
& < N\epsilon.
\end{aligned}
\end{multline}
The first inequality holds because the Fourier transform of both sides agree when $|t| > \tilde{R}$,
and the left side vanishes completely when $|t| \leq \tilde{R}$.  The second inequality
follows from expanding out the difference of $N^{\rm th}$ powers and observing that
\begin{equation*}
\max\big(\|(I-S_R)(M_yV_\epsilon R_0^+M_y^{-1})\|_{\W_\K},\,
\| M_yV_\epsilon R_0^+M_y^{-1}\|_{\W_\K}\big) \leq \pi^{-1}\|V_\epsilon\|_{\K}
\end{equation*}
while by construction $\|S_R(M_yV_\epsilon R_0^+M_y^{-1})\|_{\W_\K}
 < \epsilon(\pi^{-1}\|V\|_\K)^{1-N}$.

Based on the definition of $S_R$, it is also clear that for any $\delta < R $ and any 
$T \in \W_X$, $\| S_R(e^{i\delta \lambda} T)\|_{\W_X} \leq \| S_{R-\delta} T\|_{\W_X}$.

In order to take full advantage of~\eqref{eq:highenergy}, note that
\begin{equation*}
M_y(V_\epsilon R_0^+(\lambda^2))^N M_y^{-1} = V_\epsilon [M_y(R_0^+(\lambda^2) V_\epsilon)^{N-1}M_y^{-1}]
M_yR_0^+(\lambda^2)M_y^{-1}.
\end{equation*}
The bracketed expression on the right side has an integral kernel $K_y(\lambda,x,w)$ with the
property that $\| \int_{\R^3} |K_y(\lambda , x, w)|\,dw \|_{L^\infty_x} \les (1+|\lambda|)^{-3}$.
Multiplication by $V_\epsilon(x)$ allows us to replace $L^\infty_x$ with the Kato norm in $x$, and for
each $\lambda \in \R$ we can estimate the kernel of $M_yR_0^+(\lambda^2)M_y^{-1}$
as follows:
\begin{equation*}
\sup_w \Big\|\frac{|w-y| e^{i\lambda|w-\,\cdot\,|}}{4\pi|w-\,\cdot\,|\,|\,\cdot - y|} \Big\|_{\K^*}
\leq \pi^{-1}
\end{equation*}

This shows that not only is $\|M_y(V_\epsilon R_0^+(\lambda^2))^NM_y^{-1}\|_{\B(\K)} \les
(1+|\lambda|)^{-3}$, but the same is true if the integral kernel of this operator is replaced by its
absolute value. Consequently
\begin{equation*}
\big\|  \mc{F}\big(e^{i\delta(\,\cdot\,)} - 1) M_y(V_\epsilon R_0^+({\scriptstyle(\,\cdot\,)^2}))^NM_y^{-1}\big)(t) \big\|_{\B(\K)}
\les \int_\R |e^{i\delta \lambda}-1| (1+|\lambda|)^{-3}\,d\lambda \les \delta 
\end{equation*}
by the generalized triangle inequality.  Furthermore, the same bound in $\B(\K)$ holds if the 
integral kernel of this operator is replaced by its absolute value.  Applying the triangle inequality
once more shows that
\begin{equation*}
\big\| (I - S_{\tilde{R}}) \big((e^{i\delta\lambda}-1) M_y(V_\epsilon R_0^+)^NM_y^{-1}\big) \big\|_{\W_\K}
\les \delta \tilde{R}.
 \end{equation*}

If $\tilde{R} > NR + 1$ and $\delta < 1$, then we may conclude from~\eqref{eq:S_NR} that
the individual expressions
$\|S_{\tilde{R}}(M_y(V_\epsilon R_0^+)^NM_y^{-1})\|_{\W_\K}$ and 
$\|S_{\tilde{R}}(e^{i\delta\lambda}M_y(V_\epsilon R_0^+)^NM_y^{-1})\|_{\W_\K}$
are each less than $N\epsilon$.  It follows that
\begin{equation*}
\limsup_{\delta \to 0} \big\| (e^{i\delta \lambda} - 1)M_y(V_\epsilon R_0^+(\lambda^2))^NM_y^{-1}\big\|_{\W_\K}
\leq 2N\epsilon.
\end{equation*}

Finally, the mapping from $V \in \K$ to
$M_y(VR_0^+(\lambda^2))^NM_y^{-1} \in \W_\K$
is uniformly continuous as $V$ ranges over a bounded subset of $\K$.
Let $\eta$ denote the modulus of continuity.  Since multiplying by $e^{i\delta \lambda}$ is an isometry of $\W_\K$
for fixed $\delta$, the map from $V$ to $(e^{i\delta \lambda}-1)M_y(VR_0^+(\lambda^2))^NM_y^{-1}$
is uniformly continuous for any fixed $\delta$, and has $2\eta$ as a modulus of continuity.
Allowing $V_\epsilon$ to approach a given potential $V \in \K_0$  shows that
\begin{equation*}
\limsup_{\delta \to 0} \big\| (e^{i\delta \lambda} - 1)M_y(V R_0^+(\lambda^2))^NM_y^{-1}\big\|_{\W_\K}
\leq 2N\epsilon + 2\eta(\epsilon)
\end{equation*}
for all $\epsilon > 0$.  The details of $\eta$ depend on the size of $\| V\|_\K$, however in all cases 
one may take the limit as $\epsilon \to 0$ and conclude that~(C\ref{C2}) is satisfied.

Now all of the conditions of  Theorem~\ref{thm:inversion} are satisfied, so 
$(I + M_yVR_0^+(\lambda^2)M_y^{-1})^{-1}$ exists as an element of
$\W_\K$ for each $y \in \R^3$.
Identity~\eqref{eq:switch} allows us to conclude that 
$(I + M_yR_0^+(\lambda^2) VM_y^{-1})^{-1} \in \W_{L^\infty}$
as well.  It remains to show that these elements are uniformly
bounded with respect to $y$.  This follows directly from the claim below, whose proof is postponed to the end of this section.

\begin{lemma} \label{lem:continuity}
If $V \in \K_0$, then the mapping from $y \in \R^3$ 
to $M_yR_0^+(\lambda^2)VM_y^{-1} \in \W_{L^\infty}$ is continuous in the operator norm.
Furthermore this map is ``continuous at infinity" with a limiting value of
$R_0^+(\lambda^2)V \in \W_{L^\infty}$ as $|y| \to \infty$.
\end{lemma}

It is shown in~\cite{becgolschr} that $(I + R_0^+(\lambda^2)V)^{-1} \in \W_{L^\infty}$
(more precisely, Theorem 2 asserts that the dual operator on $\R^3$  belongs to $\W_{L^1}$).
As a corollary, the mapping from $y$ to $(I + M_yR_0^+VM_y^{-1})^{-1}$
is continuous on the one-point compactification of $\R^3$, so its range is bounded.
According to~\eqref{eq:unifWiener2}, this is precisely what is required to prove
Lemma~\ref{lem:goal}.
\end{proof}

\begin{proof}[Proof of Lemma~\ref{lem:conditions}]
Fix $R < \infty$.  By straightforward algebra,
\begin{align*}
\frac{\chi_{\{|x-z| > R\}}}{|x-z|\,|z-y|} &= 
\frac{\chi_{\{|x-z|>R\}}}{|\,|x-z| + |z-y|\,|} \Big(\frac{1}{|x-z|} + \frac{1}{|z-y|}\Big) \\
&\leq \frac{1}{\max(|x-y|,R)} \Big(\frac{1}{|x-z|} + \frac{1}{|z-y|}\Big).
\end{align*}
Both functions $|x-\,\cdot\,|^{-1}$ and $|\,\cdot\, - y|^{-1}$ have unit size
in the $\K^*$ norm, so~\eqref{eq:Kstarbound} follows from the triangle inequality.

The subsequent bound~\eqref{eq:Kbound} is nearly trivial if $V$ has compact support.
In that case, if we define $M = \max_{x \in \text{supp}\,V} |x-y|$, then the function in
question is bounded pointwise by $\frac{M}{R} V$ for all $R > M$.  Its norm is then
less than $\frac{M}{R} \|V\|_{\K}$.  The limit value as $R \to \infty$ is preserved
if one approximates a general potential $V \in \K_0$ by compactly supported
functions.

For the high-energy resolvent estimate, start by assuming $V$ is bounded with compact support.
Then $VM_y^{-1}$ maps $L^\infty(\R^3)$ to $L^2(\text{supp}\,V)$, with norm
bounded by $(1+|y|)^{-1}$.
It is well known~\cite{Ag} that $(1+|x|)^{-1}R_0^+(\lambda^2)(1+|x|)^{-1}$
belongs to $\B(L^2(\R^3))$ with norm less than $(1 + |\lambda|)^{-1}$, so in this case
$(VR_0^+(\lambda^2))^3$
maps $L^2(\text{supp}\,V)$ to itself with norm controlled by $(1+|\lambda|)^{-3}$.
We are suppressing constants which depend on the size and support of $V$.

The integral kernel of $M_yR_0^+(\lambda^2)$ has pointwise size $\frac{|x-y|}{4\pi|x-z|}$.
By Cauchy-Schwartz, this maps $L^2(\text{supp}\,V)$ to $L^\infty(\R^3)$ with norm at most
$\sup_x \frac{C|x-y|}{1+|x|} \les 1+|y|$.  Composing all of these operators together
yields the cumulative bound
\begin{equation*}
\| M_y(R_0^+(\lambda^2)V)^4M_y^{-1}\|_{\B(L^\infty)} \les (1+|\lambda|)^{-3}. 
\end{equation*}

Showing that $I + M_yVR_0^+(\lambda^2)M_y^{-1}$
is an invertible element of $\B(\K)$ for each $\lambda \in \R$ 
follows a common collection of arguments.

Let $V_\epsilon$ be an approximation of $V$ the belongs to $C^\infty_c(\R^3)$.  Since $M_y^{-1}$
maps $\K$ to $L^1(\R^3)$, we have that $V_\epsilon R_0^+(\lambda^2)M_y^{-1}$ belongs to
$\B(\K, W^{2,1}(\text{supp}\,V))$, which is a compact operator in $\B(\K, L^2(\text{supp}\,V))$.
Then $M_y$ is a bounded map from $L^2(\text{supp}\,V)$ back to $\K$.  Taking limits shows that
$M_yVR_0^+(\lambda^2)M_y^{-1}$ is compact in $\B(\K)$ so long as $V \in \K_0$.

Now suppose that $I + M_yVR_0^+(\lambda^2)M_y^{-1}$ fails to be invertible, which implies the
existence of a nontrivial $\phi \in \K$ in its nullspace by the Fredholm Alternative.  Then
$\psi = R_0^+(\lambda^2)M_y^{-1}\phi$ satisfies the eigenvalue equation
\begin{equation*}
(-\Delta + V - \lambda^2)\psi = 0.
\end{equation*}

If $\lambda = 0$, the fact that $M_y^{-1}\phi \in L^1(\R^3)$ suffices to show that
$\psi \in (1+|x|)^{s}L^2(\R^3)$ for all $s > \frac12$, which means that zero is a resonance
and possibly an eigenvalue.  This case is explicitly ruled out by the spectral assumption in
Lemma~\ref{lem:goal}.

If $\lambda \not=0$, we follow arguments for the Limiting Absorption Principle to show
the presence of an eigenvalue.  Once again let $V_\epsilon$ be a bounded compactly supported
approximation of $V$.  Then $V_\epsilon R_0^{-1}M_y^{-1}\phi \in \K$, and if
$\|V - V_\epsilon\|_\K$ is small enough then
\begin{equation*}
M_y^{-1}\phi = -(I + (V-V_\epsilon)R_0^+(\lambda^2))^{-1} V_\epsilon R_0^+(\lambda^2)M_y^{-1}\phi
\end{equation*}
is also a well defined element of $\K$, with the operator inverse in $\B(\K)$ 
existing via a convergent Neumann series.  Consequently $\psi \in L^\infty(\R^3)$.

That allows a natural pairing between $\psi \in L^\infty$ and $V\psi = -M_y^{-1}\phi \in L^1$,
with the result
\begin{equation*}
0 = \text{Re}\,\langle V\psi, \psi\rangle
 = -\text{Re}\, \langle M_y^{-1}\phi, R_0^+(\lambda^2)M_y^{-1}\phi\rangle.
\end{equation*}
Hence the Fourier transform of $V\psi$ vanishes on the sphere $|\lambda|S^2 \subset \R^3$,
and \cite[Corollary 13]{golsch2} shows that $\psi \in L^2(\R^3)$.  Since this outcome is
also forbidden by the spectral assumption in Lemma~\ref{lem:goal}, the only
remaining possibility is that $I + M_yVR_0^+(\lambda^2)M_y^{-1}$ is
invertible in $\B(\K)$.
\end{proof}

\begin{proof}[Proof of Lemma~\ref{lem:continuity}]
Once again it suffices to consider $V$ bounded with compact support, as
the uniform inequality 
$\|M_yR_0^+VM_y^{-1}\|_{\W_{L^\infty}} \leq (2\pi)^{-1} \|V\|_\K$
permits approximation of $V$.

Given two points $y_1, y_2 \in \R^3$, we must bound the size of
$\|(M_{y_2}-M_{y_1}) R_0^+VM_{y_1}^{-1}\| + 
\|M_{y_2}R_0^+V(M_{y_2}^{-1} - M_{y_1}^{-1})\|$ in $\W_{L^\infty}$.
The first term is equal to
\begin{equation*}
\sup_x \int_{\R^3} \frac{\big||x-y_2| - |x-y_1|\big|\, |V(z)|}{4\pi |x-z|\,|z-y_1|}\,dz
\les |y_2 - y_1|
\end{equation*}
where the implicit constant depends on the size and support of $V$. The second term is
equal to
\begin{equation*}
\sup_x\int_{\R^3}
\frac{|x-y_2|\,|V(z)|}{4\pi |x-z|} \Big(\frac{1}{|z-y_2|} - \frac{1}{|z-y_1|}\Big)
\,dz.
\end{equation*}

Suppose $V$ is supported in the ball $\{|z| \leq M\}$.  In one case, assume
$|y_2| > 2M$ and $|y_2 - y_1| < \frac12 M$.  Then one can make the estimate
\begin{align*}
\sup_x\int_{\R^3}
\frac{|x-y_2|\,|V(z)|}{4\pi |x-z|} \Big(\frac{1}{|z-y_2|} - \frac{1}{|z-y_1|}\Big)
\,dz 
&\les \frac{|y_2 - y_1|}{|y_2|^2}\sup_x \int_{|z| < M} \frac{|x-y_2|}{|x-z|}\,dz
\\
&\les \frac{|y_2 - y_1|}{|y_2|}.
\end{align*}
In the other case, assume $|y_1|$ and $|y_2|$ are both less than $3M$
(all instances of $y_1$ converging to $y_2$ fall into at least one of these cases).
Then use H\"older's inequality somewhat crudely to obtain the bound
\begin{align*}
\sup_x\int_{\R^3}
&\frac{|x-y_2|\,|V(z)|}{4\pi |x-z|} \Big(\frac{1}{|z-y_2|} - \frac{1}{|z-y_1|}\Big)
\,dz \\
&\les \sup_x |x-y_2| \Big\| \frac{1}{|x-\,\cdot\,|}\Big\|_{L^2(B(0,M))}
\Big\| \frac{1}{|\,\cdot\,-y_2|} - \frac{1}{|\,\cdot\,-y_1|} \Big\|_{L^2(B(0,M))}
\\
&\les \sup_x \frac{|x-y_2|}{\max(|x|,M)} |y_2 - y_1|^{1/2} \\
&\les |y_2 - y_1|^{1/2}.
\end{align*}

The limit as $|y| \to \infty$ can be handled more directly.  
The relevant bound for $R_0^+V - M_yR_0^+VR_y^{-1}$ is
\begin{equation*}
\sup_x \int_{\R^3} \frac{|V(z)|}{4\pi|x-z|}
\Big(\frac{|z-y|}{|z-y|} - \frac{|x-y|}{|z-y|}\Big)\,dz 
\leq \int_{\R^3} \frac{|V(z)|}{4\pi|z-y|}\,dz \to 0
\end{equation*}
for all $V \in \K_0$.
\end{proof}

\section{Perturbed resolvent kernel and eigenfunction bounds} \label{sec:lightcone}
While Theorem~\ref{thm:main} is nominally a statement about the wave equation,
we have observed that the main inequality~\eqref{eq:goal} in its proof is a statement
about the resolvent $R_V^+(\lambda^2)(x,y)$.  Applying Fourier inversion to~\eqref{eq:goal}
and undoing the change of variable $\lambda \to \lambda^2$
leads to a pointwise bound on resolvent kernels.
\begin{corollary} \label{cor:resolventkernel}
 Assume $V \in \K_0$, and $H = -\Delta + V$ has no eigenvalue or resonance at
zero, and no positive eigenvalues.  Then the resolvent of $H$ along the real axis
is an integral operator whose kernel $R_V^\pm(\lambda)( x, y)$ satisfies
\begin{equation}\lb{eq:pointwise}
|R_V^\pm(\lambda) (x, y)| \les \frac{1}{|x-y|}
\end{equation}
uniformly in $\lambda \geq 0$.
\end{corollary}

Recall that $R_V^+(\lambda^2)$ has a meromorphic extension in the upper half-plane, with finitely
many poles at $i\mu_j$, where $-\mu_j^2$, $j = 1,2,\ldots,J$ are the eigenvalues of
$H= -\Delta + V$.  Furthermore it is uniformly bounded as an operator on $L^2(\R^3)$
 if the imaginary part of $\lambda$ is large.  This suggests that the bulk of its Fourier
transform should be supported on the right half-line.  In fact an even stronger
statement is true, which is
equivalent to finite propagation speed of the associated wave equation.

\begin{proposition} \label{prop:cone}
Suppose $V \in \K_0$ and $H = -\Delta + V$ has no resonance at zero and
has eigenvalues only at $-\mu_j^2 < 0$, $j=1, 2, \ldots, J$.  Let $P_j$ be the
orthogonal projection onto the eigenspace of $H + \mu_j^2$.  Then for all
$\mu \in [0,\infty) \setminus \{\mu_j\}$ and $t < |x-y|$,
\begin{equation} \label{eq:cone}
\mathcal F R_V((\,\cdot+i\mu)^2)(t, x, y) = -\pi \sum_{\mu_j > \mu} \frac{e^{(\mu_j - \mu)t}}{\mu_j}P_j
\end{equation}
\end{proposition}

\begin{observation} \lb{outsidecone}
In conjunction with~\eqref{eq:stone}, this confirms that for all $|t|< |x-y|$, 
the propagator $\frac{\sin(t\sqrt{H})}{\sqrt{H}}P_c$ agrees with
$-\sum_j \frac{\sinh(\mu_j t)}{\mu_j} P_j$.
\end{observation}

\begin{proof}
The formula for the free resolvent gives 
\begin{equation*}
R_0((\lambda + i\mu)^2)(x,y) = e^{(i\lambda-\mu)|x-y|}/(4\pi|x-y|).
\end{equation*}
Examination of its Fourier transform then shows
\begin{equation}
\lim_{\mu \to \infty} \| R_0((\,\cdot\,+i\mu)^2)V\|_{\W_{L^\infty}} = 
\lim_{\mu \to \infty} \sup_{x\in\R^3} \int_{\R^3} \frac{e^{-\mu|x-y|}V(y)}{4\pi|x-y|}\,dy = 0.
\end{equation}
The last limit converges at a rate $1/\mu^2$ if $V$ is bounded and compactly supported,
and is preserved under taking limits to $V \in \K_0$.  Therefore, for sufficiently large $\mu_0<\infty$,
one can compute $(I + R_0((\,\cdot\,+i\mu_0)^2)V)^{-1}$ in $\W_{L^\infty}$ via
a convergent power series.  

The Fourier transform of $R_0((\,\cdot\, + i\mu_0)^2)V$ is a measure supported on the ``forward light cone"
$\{t = |x-y|\}$.  The Fourier transform of each higher power of this object is an iterated
convolution of the measure, which is supported in $\{t \geq |x-y|\}$ by the triangle inequality.
Hence its Wiener algebra inverse, $[\mathcal F((I + R_0(\,\cdot\,+i\mu_0)^2)V)]^{-1}$, is supported here as well,
and so is the Fourier transform of 
$R_V((\,\cdot\,+i\mu_0)^2) = ((I + R_0(\,\cdot\,+i\mu_0)^2)V)^{-1}R_0((\,\cdot\,+i\mu_0)^2)$.


Write out the Fourier transform of $R_V((\,\cdot\,+i\mu_0)^2)(x,y)$ as the contour integral
\begin{equation}\label{eq:contour}
\mathcal F R_V((\,\cdot\,+i\mu_0)^2)(t,x,y) = \int_{\R+i\mu_0}e^{-i t z}e^{-\mu_0 t} R_V(z^2)(x,y)\,dz
\end{equation}
and similarly,
\begin{equation*}
\mathcal F R_V((\,\cdot\,+i\mu)^2)(t,x,y) = \int_{\R+i\mu}e^{-i t z}e^{-\mu t} R_V(z^2)(x,y)\,dz.
\end{equation*}

According to the Stone formula for spectral measure, the residue of $R_V(z)$ at each pole $z =-\mu_j^2$
is the projection $-P_j$.  Then the residue of $R_V(z^2)$ at each pole $z = i\mu_j$ is precisely
$(-2 i\mu_j)^{-1} P_j$.
If we can shift the contour of integration from the horizontal line $\R + i\mu_0$ to the
line $\R + i\mu$, the result would be
\begin{equation*}
\mathcal FR_V((\,\cdot\,+i\mu)^2)(t,x,y) 
= e^{-\mu t}\Big[e^{\mu_0 t} \mathcal F R_V((\,\cdot\,+i\mu_0)^2)(t,x,y) 
-\pi  \sum_{\mu <\mu_j < \mu_0} \frac{e^{\mu_j t}}{\mu_j }P_j\Big]
\end{equation*}

To justify shifting the contour, we first claim that $R_V((\lambda+i\mu)^2)$ is bounded in
a useful way for large $\lambda$.  Note that if $V$ is bounded with compact support,
the statement and proof of~\eqref{eq:highenergy}
in Lemma~\ref{lem:conditions} remain valid for estimating the operator norm of
$M_y(R_0^+(\lambda+i\mu)^2V)^4M_y^{-1}$,
with constants that are independent of $\mu$.  It follows by an approximation argument that for
$V \in \K_0$, the limit 
\begin{equation*}
\lim_{\lambda \to\pm \infty} \|M_y(R_0^+(\lambda+i\mu)^2V)^4M_y^{-1}\|_{\mathcal B(L^\infty)} = 0
\end{equation*}
is uniform in $\mu$.

As a consequence, for $\lambda$ sufficiently large the perturbation series of $R_V((\lambda + i\mu)^2)$
converges absolutely and has a pointwise bound $|R_V((\lambda+i\mu)^2)(x,y)| \leq C|x-y|^{-1}$.
Now the contour in~\eqref{eq:contour} can be moved by multiplying both sides of the equation by 
the analytic function $e^{-\epsilon(\lambda+i\mu)^2}$ then taking the limit as $\epsilon$ goes to zero.
\end{proof}

It is well known that the null space of $H+ \mu_j^2$ consists entirely of functions
with exponential decay at infinity.  The following quantitative bound will be
extremely useful for computing pointwise behavior of spectral multipliers in the next section.

\begin{lemma} \label{lem:projections}
For each eigenvalue $-\mu_j^2<0$ of the operator $H = -\Delta + V$, 
the projection $P_j$ onto its eigenspace satisfies $|P_j(x,y)| \leq C\mu_j e^{-\mu_j|x-y|}/|x-y|$.

Furthermore, since each eigenfunction of $H$ is a bounded function, it is also true that
$|P_j(x,y)| \leq C \mu_j^{2-\alpha} e^{-\mu_j|x-y|}/|x-y|^{-\alpha}$ for all $0 \leq \alpha \leq 1$.
\end{lemma}

\begin{proof}
Given a finite number of distinct positive numbers $\mu_j$, $j = 1,\ldots, J$,
define a function $F:\R^J \to \R$  by 
\begin{equation*}
F(a_1, \ldots, a_J) = \int_{-\infty}^0 \Big|\sum_{j=1}^J a_je^{\mu_j t}\Big|\,dt.
\end{equation*}
This is a nonnegative continuous function on $\R^J$, homogeneous of order 1, and it is not zero except
when all $a_j = 0$.  By compactness, the infimum $\min\{F(a_1,\ldots,a_J): \max_j |a_j|= 1\}$ is
a number $c>0$.

The combination of Proposition~\ref{prop:cone} (with $\mu = 0$) and the bound~\eqref{eq:goal} from
Theorem~\ref{thm:main} yields
\begin{equation*}
\int_{-\infty}^{|x-y|} \Big| \sum_{j=1}^J \frac{P_j(x,y)}{\mu_j}e^{\mu_j t}\Big|\,dt \leq \frac{C}{|x-y|}.
\end{equation*}
The property of function $F$ discussed above implies that 
\begin{equation*}
\max_j \frac{|P_j(x,y)|e^{\mu_j|x-y|}}{\mu_j} \leq \frac{C}{c|x-y|}.
\end{equation*}
This proves the first claim.  Boundedness of eigenfunctions implies that there is a constant $\tilde{C}$
so that $|P_j(x,y)| \leq \tilde{C}\mu_j \min(1,\,|x-y|^{-1}e^{-\mu_j|x-y|})$ which is dominated in turn by
$\tilde{C}\mu_j |x-y|^{-\alpha}e^{-\mu_j|x-y|}$.  By assumption, the values $\mu_j$ are a finite set of
positive numbers.  So one may replace $\mu_j$ by $\mu_j^{2-\alpha}$ at the cost of changing the constant once more.

\end{proof}

\section{$L^p$ and pointwise kernel bounds for Mihlin multipliers}
Lemma~\ref{lem:goal} is a rather direct statement about how
the spectral measure of $H = -\Delta + V$ differs from 
that of the Laplacian on $\R^3$.  As a result it provides information
about a wide range of elements in the functional calculus of $H$
beyond the wave propagator $\frac{\sin(t\sqrt{H})}{\sqrt{H}}$.
We show that two of the most common criteria for $L^p$ boundedness
of radial Fourier multipliers $m(|\nabla|)$ are also sufficient
to ensure boundedness of the corresponding operator $m(\sqrt{H})$.

\begin{theorem} \label{thm:multiplier3}
 Assume $V \in \mc K_0 \cap L^{3/2, \infty}$, and $H = -\Delta + V$ has no eigenvalue or resonance at
zero, and no positive eigenvalues.  Let $\phi$ be a $C^\infty$ function supported on
$[\frac12, 2]$ such that $\{\phi(2^{-k}\lambda)\}_{k\in\mathbb Z}$ forms a partition of unity for
$\R^+$.

Suppose $m: (0,\infty) \to \mathbb C$ satisfies
\begin{equation} \label{eq:HsCondition}
\sup_{k\in \Z} \|\phi(\lambda)m(2^k\lambda)\|_{H^s} = M_s < \infty \text{ for some } s> \frac32.
\end{equation}
Then $m(\sqrt{H}) := m(\sqrt{H P_c})$ is a singular integral operator that is bounded on
$L^p(\R^3)$ for all $1 < p < \infty$, with operator norm less than $C_{s,p} M_s$.

If $V \in \mc K_0$ and~\eqref{eq:HsCondition} holds for some $s>2$, then $m(\sqrt{H})$ has an integral kernel 
$K(x,y)$ that is bounded pointwise by
\begin{equation}
|K(x,y)| \les \frac{M_s}{|x-y|^3}.
\end{equation}
\end{theorem}


The proof of both claims in the theorem will use the following lemmas:

\begin{lemma} \label{lem:conicalbound}
 Assume $V \in \K_0$, and $H = -\Delta + V$ has no eigenvalue or resonance at
zero, and no positive eigenvalues.  Then the Fourier transform of $R_V^+(\lambda^2)$
satisfies the bound
\begin{equation} \label{eq:conicalbound}
\int_{|x-y| \leq t} \big| \mathcal F(R_V^+({\scriptstyle (\,\cdot\,)^2}))(t, x, y)\big|\,dx \les t
\end{equation}
uniformly in $y \in \R^3$, $t > 0$.
\end{lemma}

\begin{proof}
 Use resolvent identities to write
\begin{equation*}
R_V^+(\lambda^2) = R_0^+(\lambda^2) - R_0^+(\lambda^2) VR_V^+(\lambda^2).
\end{equation*}
Recalling that $R_0^+(\lambda^2)(x,y) = \frac{e^{i\lambda|x-y|}}{4\pi|x-y|}$,
there is an explicit formula
\begin{equation*}
\mathcal F(R_0^+(\lambda^2))(t, x, y) = \frac{\delta(t-|x-y|)}{4\pi|x-y|}.
\end{equation*}
which satisfies~\eqref{eq:conicalbound}.

It is a consequence of~\eqref{eq:goal}, or \cite[Theorem 1]{becgol}, that
for each fixed $y \in \R^3$, the integral kernel $\mathcal F(VR_V^+(\lambda^2))(t,x,y)$ is a measure
in $\R^{1+3}$ with finite total variation.  It follows that 
\begin{multline*} 
\int_{|x-y| \leq t} \big| \mathcal F (R_V^+(\lambda^2)- R_0^+(\lambda^2))(t, x ,y)\big|\,dx
\\
\begin{aligned}
&
\leq \int_{|x-y| \leq t} \int_{-\infty}^t  \int_{\R^3}
\big|\mathcal F(R_0^+(\lambda^2))(t-s, x,w)\big|\,
\big|\mathcal F(VR_V^+(\lambda^2))(s,w,y)\big|\,dw ds dx 
\\
&\leq \sup_{s<t} \sup_{w\in\R^3} \int_{|x-y| \leq t} \big|\mathcal F (R_0^+(\lambda^2))(t-s, x ,w)\big|\,dx \\
&= \sup_{s<t} \sup_{w\in\R^3} (4\pi(t-s))^{-1} \text{Area}\big(\{|x-w| = t-s\}\cap \{|x-y| \leq t\}\big)\\
&\les \sup_{s<t} \,(t-s)^{-1}\min((t-s)^2, t^2) \,\leq\, t. 
\end{aligned}
\end{multline*}
\end{proof}

\begin{lemma} \label{lem:LpTransfer}
Given a function $g:\R \to \C$, the inequality
\begin{equation} \label{eq:LpTransfer}
\Big\| |x-y|^{\beta} \int_{t \geq |x-y|} g(t) \mathcal F(R_V(\lambda^2))(t,\,\cdot\,, y)\, dt \Big\|_{L^p(\R^3)}
\les \big\| |t|^{\frac{2}{p}-1+\beta} g(t) \big\|_{L^p(\R_+)}
\end{equation}
holds uniformly for $1 \leq p \leq \infty$, $y \in \R^3$, and $\beta \geq 0$.
\end{lemma}

\begin{proof}
Lemma~\ref{lem:conicalbound} immediately implies the $p=1$, $\beta = 0$ case.   Theorem~1 of~\cite{becgol} states that
\begin{equation*}
\sup_{x,y} \int_\R |t| \Big|\frac{\sin(t\sqrt{H})P_c}{\sqrt{H}}(t, x, y)\,dt\Big| < \infty
\end{equation*}
and it is proved via the bound
\begin{equation*}
\sup_{x,y} \int_\R \big|t \mathcal F(R_V^+({\scriptstyle (\,\cdot\,)^2}))(t, x, y)\,dt\big| < \infty.
\end{equation*}

This is equivalent to the $p=\infty$, $\beta = 0$ case, and the rest of the $\beta = 0$ cases follow by complex interpolation.  The $\beta \geq 0$ cases follow immediately as well, since $\frac{|x-y|}{t} \leq 1$ within the domain of integration.
\end{proof}

Fix a smooth cutoff function $\chi \in C^\infty_c(\R)$ such that $\sum_{k \in \Z} \chi(2^{-k} t) = \chi_{(0, \infty)}(t)$ and $\supp \chi \subset [2/3, 3]$.

\begin{definition}
The Littlewood-Paley square function for the perturbed Hamiltonian is
$$
[S_H f](x) = \bigg(\sum_{k \in \Z} |[\chi(2^{-k} \sqrt H) f](x)|^2\bigg)^{1/2}.
$$
\end{definition}

By a standard argument, the square function is comparable to the original function. Square functions corresponding to different choices of dyadic cutoffs are all comparable in the $L^p$ norms, $1<p<\infty$.

\begin{theorem} Given $H=-\Delta+V$ and assuming that $V \in \mc K_0 \cap L^{3/2, \infty}$, for each $p \in (1, \infty)$
$$
\|S_H f\|_{L^p} \les_p \|f\|_{L^p} \les_p \|S_H f\|_{L^p}.
$$
\end{theorem}
\begin{proof} We use Kintchine's inequality, see \cite{wolff}. Let $(\epsilon _k)_{k \in \Z}$ be a sequence of independent identically distributed random variables, with $\P(\epsilon_k=\pm 1)=1/2$, and consider a sequence of complex numbers $(y_k)_{k \in \Z}$. Then for each $q \in (0, \infty)$
$$
\|(y_k)\|_{\ell^2_k} \les_q \bigg(\E \bigg|\sum_{k \in \Z} \epsilon_k y_k\bigg|^q\bigg)^{1/q} \les_q \|(y_k)\|_{\ell^2_k}.
$$

We apply this bound to $y_k = [\chi(2^{-k} \sqrt H) f](x)$, pointwise for each $x \in \R^3$. Then $[S_H f](x)$ is bounded by the $q$-th moment of $m_\epsilon(\sqrt H)(x)$, where $m_\epsilon$ is the random Mihlin multiplier
$$
m_\epsilon(\lambda) = \sum_{k \in \Z} \epsilon_k \chi(2^{-k} \lambda).
$$
These random multipliers uniformly fulfill the hypotheses of Theorem \ref{thm:multiplier3}, hence
$$
\|m_\epsilon(\sqrt H) f\|_{L^p} \les \|f\|_{L^p},
$$
uniformly over all choices of $(\epsilon_k)$. Setting e.g.\;$q=1$ above, we obtain that for $p \in (1, \infty)$
$$
\|S_H f\|_{L^p} \les \|\E |m_\epsilon(\sqrt H)|\|_{L^p} \les \E \|m_\epsilon(\sqrt H)\|_{L^p} \les \E \|f\|_{L^p} = \|f\|_{L^p}.
$$
The other inequality follows in a standard manner from the previous one, by duality. We repeat the usual proof here. Consider a sequence of slightly larger cutoffs $\tilde \chi \in C^\infty_c(\R)$, such that $\tilde \chi = 1$ on $\supp \chi$ and $\supp \tilde \chi \subset [1/2, 4]$. Fix $p \in (1, \infty)$. Then for any two functions $f \in L^p$, $g \in L^{p'}$
$$\begin{aligned}
|\langle f, g \rangle| &= \sum_{k \in \Z} \langle f, \chi(2^{-k} \sqrt H) g \rangle = \sum_{k \in \Z} \langle \tilde \chi(2^{-k} \sqrt H) f, \chi(2^{-k} \sqrt H) g \rangle \\
&\leq \int_{\R^3} \sum_{k \in \Z} |[\tilde \chi(2^{-k} \sqrt H) f](x) [\chi(2^{-k} \sqrt H) g](x)| \dd x \\
&\leq \langle \tilde S_H f, S_H g \rangle \les \|\tilde S_H f\|_{L^p} \|S_H g\|_{L^{p'}} \les \|f\|_{L^p} \|S_H g\|_{L^{p'}},
\end{aligned}$$
because the modified Littlewood-Paley square function $\tilde S_H$ is comparable to the standard one, hence to $f$ itself, in the $L^p$ norm.

Choosing $f$ such that $\|f\|_{L^p} \les 1$ and $\langle f, g \rangle = \|g\|_{L^{p'}}$, it follows that for $p' \in (1, \infty)$
$$
\|g\|_{L^{p'}} \les \|S_H g\|_{L^{p'}},
$$
which is the desired conclusion.
\end{proof}

\begin{proof}[Proof of Theorem~\ref{thm:multiplier3}]
By linearity, it suffices to prove the statements when $M_s = 1$ with a bound
that does not depend on $m$ in any other way.
First use the Stone formula for the spectral measure of $H$ to write

$$
m(\sqrt{H}) = (\pi i)^{-1}\int_{-\infty}^\infty \lambda\tilde{m}(\lambda)R_V^+(\lambda^2)\,d\lambda.
$$
This time it will be convenient to use the resolvent identity 
$$
R_V^+(\lambda^2) = R_0^+(\lambda^2) - R_V^+(\lambda^2)VR_0^+(\lambda^2)
$$
in the right-hand integral.  The first term is just the spectral representation of $m(\sqrt{-\Delta})$, which is a Calder\'on-Zygmund operator and therefore bounded on $L^p(\R^3)$, $1 < p < \infty$.  We prepare the second term by breaking the spectral multipliers into compact intervals.
\begin{align*}
m(\sqrt{H}) - m(\sqrt{-\Delta}) 
&= \sum_{k=-\infty}^\infty \phi(2^{-k} \sqrt{H})m(\sqrt{H})P_c - \phi(2^{-k}(-\Delta))m(\sqrt{-\Delta}) \\
&= \sum_{k=-\infty}^\infty -(\pi i)^{-1} \int_{-\infty}^\infty \phi(2^{-k}|\lambda|) \lambda \tilde{m}(\lambda)
R_V^+(\lambda^2)VR_0^+(\lambda^2)\,d\lambda \\
&= \sum_{k=-\infty}^\infty T_k,
\end{align*}
with the summation on the right side of the top line converging in the strong operator topology of $\B(L^2(\R^3))$. 

Let $g_k(\lambda)= \phi(2^{-k}|\lambda|)\lambda \tilde{m}(\lambda)$
so that $\sum_k g_k(\lambda) = \lambda \tilde{m}(\lambda)$.
The two main estimates on the Fourier transform of $g_k$ are
\begin{equation}\label{eq:dyadicbounds}
|\check{g}_k(t)| \les 2^{2k} \quad \text{and} \quad 
\|\check{g}_k(\,\cdot\,)\langle 2^k\,\cdot\,\rangle^s \|_{L^2(\R)} \les 2^{\frac32 k}.
\end{equation}

The integral kernel of $T_k$ can be written out as
\begin{equation*}
T_k(x,y) = -(4\pi^2 i)^{-1} \int_{-\infty}^\infty \int_{\R^3}
g_k(\lambda) R_V^+(\lambda^2)(x, z) V(z) \frac{e^{i\lambda |z-y|}}{|z-y|}\,dz d\lambda
\end{equation*}
Compact support of $g_k$ and the resolvent kernel bound~\eqref{eq:pointwise} show that the integral converges absolutely when $x \not= y$.  After reversing the order of integration and applying the Parseval identity, one obtains
\begin{align*}
T_k(x,y) &= -(4\pi^2i)^{-1} \int_{\R^3} \int_{-\infty}^\infty
\check{g}_k(t+|z-y|) \mathcal{F}(R_V^+({\scriptstyle (\,\cdot\,)^2}))(t, x, z) \frac{V(z)}{|z-y|}\,dt dz \\
&= -(4\pi^2i)^{-1} \int_{\R^3} \int_{|x-z|}^\infty
\check{g}_k(t+|z-y|) \mathcal{F}(R_V^+({\scriptstyle (\,\cdot\,)^2}))(t, x, z) \frac{V(z)}{|z-y|}\,dt dz \\
& \quad -(4\pi^2i)^{-1} \int_{\R^3} \int_{-\infty}^{|x-z|}
\check{g}_k(t+|z-y|) \mathcal{F}(R_V^+({\scriptstyle (\,\cdot\,)^2}))(t, x, z) \frac{V(z)}{|z-y|}\,dt dz \\
&= T_{k,1}(x,y) + T_{k,2}(x,y) 
\end{align*}

The $t$ integral in $T_{k,1}$ is exactly the type considered in Lemma~\ref{lem:LpTransfer}. In this way it is controlled indirectly by the estimates for $\check{g}_k$ in~\eqref{eq:dyadicbounds} and H\"older's inequality.  More precisely, for $1 \leq p \leq 2$,
\begin{align*}
\|T_{k,1}(\,\cdot\, , y)\|_{L^p(\R^3)}
&\les \int_{\R^3} \big\| t^{\frac2p - 1} \check{g}_k(t+|z-y|)\big\|_{L^p(\R_+)} \frac{V(z)}{|z-y|}\,dz \\
&\les \int_{\R^3} 2^{\frac32 k} \Big\| \frac{t^{2/p - 1}}{\langle 2^k(t + |z-y|)\rangle^s}\Big\|_{L^{\frac{2p}{2-p}}(\R_+)} \frac{V(z)}{|z-y|}\,dz.
\end{align*}
Direct inspection of the norm on $\R_+$ shows that it is concentrated on the interval $t \leq 2^{-k}$ if $2^k|z-y| < 1$ and on the interval $t \leq |z-y|$ if $2^{k}|z-y| > 1$.  The tail is integrable since $s - \frac2p + 1 > \frac1p - \frac12$.  The resulting bound is
\begin{equation*}
\|T_{k,1}(\,\cdot\, , y)\|_{L^p(\R^3)}
\les \int_{\R^3} \frac{2^{(3 - \frac3p)k}}{\langle2^k|z-y|\rangle^{s + \frac32- \frac3p}} \frac{V(z)}{|z-y|}\,dz,
\end{equation*}
and since $3 - \frac3p$ and $s - \frac32$ are both positive, this can be summed over $k$ as well.
\begin{equation}
\sum_{k \in \Z} \|T_{k,1}(\,\cdot\, , y)\|_{L^p(\R^3)}
\les \int_{\R^3} \frac{V(z)}{|z-y|^{4-\frac3p}}\,dz.
\end{equation}

Thus for $p$ in the range  $1 < p \leq 2$,
\begin{equation*}
\Big| \int_{\R^3} \int_{\R^3} \frac{V(z)}{|z-y|^{4 - \frac3p}} f(y)\,dy dz \Big| = \Big|\langle V \ast \frac 1 {|x|^{4-3/p}}, f\rangle\Big| \les \|V\|_{L^{3/2, \infty}} \Big\|\frac 1{|x|^{4-\frac3p}}\Big\|_{L^{\frac {3p}{4p-3}, \infty}} \|f\|_{L^{p, 1}(\R^3)},
\end{equation*}
as $L^{3/2,\infty} \ast L^{3p/(4p-3), \infty} \mapsto L^{p', \infty}$.  In conclusion, $\sum\limits_{k\in \Z} T_{k,1}$ describes a bounded operator from $L^{p, 1}(\R^3)$ to $L^p(\R^3)$.

The analysis of $T_{k,2}$ is more explicit. Observe that
\begin{equation*}
\mathcal{F}(R_V^+({\scriptstyle (\,\cdot\,)^2}))(t) = \mathcal F \sum_{j=1}^J \frac {P_j(x, z)}{-\mu_j^2-\lambda^2} = -\frac 1 {4\pi} \sum_{j=1}^J P_j(x, z) \frac {e^{-\mu_j|t|}}{\mu_j}.
\end{equation*}
Then, since $\check g_k$ are odd functions,
\begin{equation*}\begin{aligned}
&\Big| \int_{-\infty}^{|x-z|} \check{g}_k(t+|z-y|) \mathcal{F}(R_V^+({\scriptstyle (\,\cdot\,)^2}))(t, x, z) \,dt \Big|
= \\
&= \frac 1 {4\pi} \Big| \int_{-\infty}^{|x-z|} \check{g}_k(t+|z-y|) \sum_{j=1}^J P_j(x, z) \frac {e^{-\mu_j |t|}} {\mu_j} \,dt \Big|\\
&\les \sum_{j=1}^J \frac 1 {|x-z|} \Big| \int_{-\infty}^{-|x-z|} \check{g}_k(t+|z-y|) e^{\mu_j t} \,dt\Big|\\
&\les \frac{1}{|x-z|} \sum_{j=1}^J \Big|\Big(\frac{g_k(\lambda)}{\lambda - i\mu_j}\Big)\check{\phantom{i}}(|x-z| + |z-y|)\Big|.
\end{aligned}\end{equation*}
By integrating in spherical coordinates centered at $z$, the $L^p(\R^3)$ norm of the right-hand side with respect to $x$
is exactly 
$\sum_j \big\| t^{\frac2p - 1} \big(\frac{g_k(\lambda)}{\lambda - i\mu_j}\big)\check{\phantom{i}}(t + |z-y|)\big\|_{L^p(\R_+)}$.

Observe that the function $(\lambda -i\mu_j)^{-1}$ is smooth and roughly constant on each
interval $\lambda \in \pm [2^{k-1}, 2^{k+1}]$ with size $\min(\mu_j^{-1},2^{-k}) \leq\mu_j^{-1}$. 

Thus each function $\frac{g_k}{\lambda - i\mu_j}$ has the same $H^s$ control as $g_k$ itself, so its Fourier transform is subject to the bounds~\eqref{eq:dyadicbounds}. Every estimate that is made for $\tilde{m}$ and its Fourier transform holds
for $\frac{\tilde{m(\lambda)}}{\lambda - i\mu_j}$ with a constant possibly $\max{\mu_j^{-1}}$ times as large. It follows that
\begin{equation*}
\|T_{k,2}(\,\cdot\, , y)\|_{L^p(\R^3)}
\les \int_{\R^3} 2^{\frac32 k} \Big\| \frac{t^{2/p - 1}}{\langle 2^k(t + |z-y|)\rangle^s}\Big\|_{L^{\frac{2p}{2-p}}(\R_+)} \frac{V(z)}{|z-y|}\,dz
\end{equation*}
and the argument that $\sum\limits_{k\in\Z} T_{k,2}$ maps $L^{p, 1}(\R^3)$ to $L^{p}(\R^3)$ for $1 < p \leq 2$ now proceeds identically to the discussion of $T_{k,1}$ above.

Recall once more that $m(\sqrt{H}) = m(\sqrt{-\Delta}) + \sum_k T_{k,1} + \sum_k T_{k,2}$.  We now know that the right-hand side is bounded as an operator from $L^{p,1}(\R^3)$ to $L^p(\R^3)$ for any $1 < p \leq 2$ and the left side is bounded on $L^2(\R^3)$.  Marcinkiewicz interpolation shows that $m(\sqrt{H})$ is also bounded on $L^p(\R^3)$ for all $1 < p \leq 2$, and by duality the range is extended to all $1 < p < \infty$.

In order to prove the $|x-y|^{-3}$ bound for the integral kernel of $m(\sqrt H)$, we start from
\begin{align*}
m(\sqrt{H}) &= (2\pi i)^{-1} \int_0^\infty m(\sqrt{\lambda})
(R_V^+(\lambda) - R_V^-(\lambda))\,d\lambda \\
&= (\pi i)^{-1} \int_0^\infty \lambda m(\lambda) 
(R_V^+(\lambda^2) - R_V^-(\lambda^2))\,d \lambda \\
&= (\pi i)^{-1} \int_{-\infty}^\infty \lambda \tilde{m}(\lambda) R_V^+(\lambda^2)\,d\lambda,
\end{align*}
where $\tilde{m}$ is the even extension of $m$ to the full line.
Then by Parseval's identity, the kernel of $m(\sqrt{H})$ can be re-written once more as
\begin{equation} \label{eq:Stonemultiplier}
K(x,y) = (\pi i)^{-1}\int_{-\infty}^\infty (\lambda \tilde{m})\widecheck{\phantom{i}}(t) 
\mathcal F R_V^+({\scriptstyle (\,\cdot\,)^2})(t,x,y)\,dt.
\end{equation}
We split this into two parts,
\begin{align} \label{eq:conemultiplier}
K_1(x,y) &= (\pi i)^{-1}\int_{t \geq |x-y|} (\lambda \tilde{m})\widecheck{\phantom{i}}(t) 
\mathcal F R_V^+({\scriptstyle (\,\cdot\,)^2})(t,x,y)\,dt \\
K_2(x,y) &= (\pi i)^{-1}\int_{t < |x-y|} (\lambda \tilde{m})\widecheck{\phantom{i}}(t) 
\mathcal F R_V^+({\scriptstyle (\,\cdot\,)^2})(t,x,y)\,dt. \label{eq:boundstatesmultiplier}
\end{align}

Estimates for $K_1(x,y)$ generally follow from Lemma~\ref{lem:LpTransfer}.

We use Proposition~\ref{prop:cone} to write out
$\mathcal F R_V^+({\scriptstyle (\,\cdot\,)^2})(t)$ as a sum of exponential functions.
Observe that
\begin{equation*}
\mathcal F^{-1} \big(e^{\mu_j t}{\bf 1}_{t <|x-y|}\big)(\lambda) 
= \frac{e^{(i\lambda + \mu_j)|x-y|}}{2\pi i(\lambda -i\mu_j)}.
\end{equation*}
After undoing Parseval's identity in~\eqref{eq:boundstatesmultiplier}, the quantity we need to bound is
\begin{equation}\label{eq:K2bound}
\begin{aligned} 
\bigg|\sum_{j=1}^J \frac{e^{\mu_j|x-y|}P_j(x,y)}{2\pi\mu_j}
 &\int_\R \frac{\lambda \tilde{m}(\lambda)}{\lambda - i\mu_j}e^{i\lambda|x-y|}\,d\lambda\bigg|
\\
&\les \frac{1}{|x-y|} \sum_{j=1}^J 
\Big|\int_\R \lambda \Big(\frac{\tilde{m}(\lambda)}{\lambda - i\mu_j}\Big)e^{i\lambda|x-y|}\,d\lambda\Big| \\
&= \frac{1}{|x-y|} \sum_{j=1}^J \Big| \Big(\frac{\lambda \tilde{m}}{\lambda - i\mu_j}\Big)\widecheck{\phantom{i}}(|x-y|)\Big|
\end{aligned}
\end{equation}

If  $m$ satisfies~\eqref{eq:HsCondition} with the stronger assumption $s > 2$,
then the standard dyadic proof of the Mihlin-H\"ormander theorem
shows that $|(\lambda \tilde{m})\widecheck{\phantom{i}}(t)|
\les |t|^{-2}$ away from $t=0$.  The same is true for the inverse Fourier transform of each 
$\frac{\lambda \tilde{m}}{\lambda - i\mu_j}$.  Then $|K_1(x,y)| \les |x-y|^{-3}$ by applying Lemma~\ref{lem:LpTransfer}
with $p=\infty$, $\beta = 3$.

Meanwhile $|K_2(x,y)| \les |x-y|^{-3}$ by direct inspection of~\eqref{eq:K2bound}.
\end{proof}

\section{Weak-type $(1, 1)$ bounds for Mihlin multipliers}

While Theorem~\ref{thm:multiplier3} hints at Mihlin multipliers $m(\sqrt{H})$ being Calder\'on-Zygmund operators thanks to the pointwise kernel bound $|K(x,y)| \les |x-y|^{-3}$, it is not clear that they satisfy the requisite cancellation conditions.  In this section we state a weak-type (1,1) bound that holds for the same multipliers as in the second part of Theorem~\ref{thm:multiplier3}, with one additional assumption on the potential $V$.

\begin{theorem} \lb{thm:weakL1}
Suppose $H = -\Delta + V$, where $V \in \mc K_0$ has the property
\begin{equation} \label{eq:variantKato}
\|V\|_{\tilde \K} := \sup_{x \in \R^3} \sum_{\ell \in \mathbb Z} \|\chi(|z-x| \sim 2^\ell) V(z)\|_{\K_z} < \infty.
\end{equation}
Further suppose that $H$ has no eigenvalue or resonance at
zero, and no positive eigenvalues.
 Let $\phi$ be a $C^\infty$ function supported on
$[\frac12, 2]$ such that $\{\phi(2^{-k}\lambda)\}_{k\in\mathbb Z}$ forms a partition of unity for
$\R^+$.  Suppose $m: (0,\infty) \to \mathbb C$ satisfies
\begin{equation*}
\sup_k\|\phi(\lambda)m(2^k\lambda)\|_{H^s} = M_s < \infty \text{ for some } s> 2.
\end{equation*}
Then $m(\sqrt{H}) \in \B(L^1, L^{1,\infty})$ and $m(\sqrt{H}) \in \B(L^p)$, $1<p<\infty$. 
\end{theorem}

Before proving this theorem, note that not only $\dot W^{1, 1} \subset \K$, as shown in \cite{becgol}, but in fact $\dot W^{1, 1} \subset \tilde \K$ as well.
\begin{lemma}\lb{w11embed} $\dot W^{1, 1} \subset \tilde \K$.
\end{lemma}
\begin{proof} Consider a cutoff function $\chi \in C^\infty_c(\R)$ such that $\supp \chi \subset [3/4, 3]$ and $\chi(x) = 1$ when $x \in [1, 2]$. In light of the fact that $\dot W^{1, 1} \subset \K$, it suffices to show that
$$
\sum_{\ell \in \Z} \|\chi(|z-x| \sim 2^\ell) V(z)\|_{\dot W^{1, 1}_z} \les \|V\|_{\dot W^{1, 1}}.
$$
We prove this separately for even and for odd $\ell$. With no loss of generality, consider the even case. Since the even-numbered cutoffs have disjoint support, for this case the claim reduces to
$$
\bigg\|\sum_{\ell \in 2\Z} \chi(|z-x| \sim 2^\ell) V(z)\bigg\|_{\dot W^{1, 1}_z} \les \|V\|_{\dot W^{1, 1}}.
$$
By the Leibniz rule, the gradient of this expression is
$$
\nabla_z \bigg(\sum_{\ell \in 2\Z} \chi(|z-x| \sim 2^\ell) \bigg) V(z) + \bigg(\sum_{\ell \in 2\Z} \chi(|z-x| \sim 2^\ell) \bigg) \nabla V(z).
$$
The gradient of the sum of these cutoffs is of size $\frac 1 {|z-x|}$, due to scaling, so the first term is integrable, as $V \in \dot W^{1, 1} \subset \K$. The second term is integrable as well, because $\nabla V \in L^1$ and the sum of the cutoffs is bounded.

This proves the claim for the sum over the even indices and one proceeds similarly for the odd indices.
\end{proof}

\begin{observation} Note that $L^{3/2, 1} \not \subset \tilde K$, despite the fact that $L^{3/2, 1} \subset \K$. An example is given by the indicator function of the union of countably many balls: $\chi_{\bigcup_{k \geq 0} B_k}$, where $B_k$ is at distance $2^k$ from the origin and such that $\mu(B_k) \in \ell^1_k$, $\mu(B_k)^{2/3} \not \in \ell^1_k$.
\end{observation}

\begin{proof}[Proof of Theorem \ref{thm:weakL1}]
Once again we use the Stone formula to write
$$
m(\sqrt{H}) = (\pi i)^{-1}\int_{-\infty}^\infty \lambda\tilde{m}(\lambda)R_V^+(\lambda^2)\,d\lambda.
$$
This time it will be convenient to use the resolvent identity 
$$
R_V^+(\lambda^2) = R_0^+(\lambda^2) - R_V^+(\lambda^2)VR_0^+(\lambda^2)
$$
in the right-hand integral.  The first term is just the spectral representation of $m(\sqrt{-\Delta})$, which is a Calder\'on-Zygmund operator and is of weak-type (1,1).  We prepare the second term by breaking the spectral multipliers into compact intervals.
\begin{align*}
m(\sqrt{H}) - m(\sqrt{-\Delta}) 
&= \sum_{k=-\infty}^\infty \phi(2^{-k} \sqrt{H})m(\sqrt{H})P_c - \phi(2^{-k}(-\Delta))m(\sqrt{-\Delta}) \\
&= \sum_{k=-\infty}^\infty -(\pi i)^{-1} \int_{-\infty}^\infty \phi(2^{-k}|\lambda|) \lambda \tilde{m}(\lambda)
R_V^+(\lambda^2)VR_0^+(\lambda^2)\,d\lambda \\
&= \sum_{k=-\infty}^\infty T_k,
\end{align*}
with the summation on the right side of the top line converging in the strong operator topology of $\B(L^2(\R^3))$. 

For each $k$, the kernel of $T_k$ is obtained by composing the operators $R_V^+(\lambda^2) V R_0^+(\lambda^2)$
inside the integral.  Let $g_k(\lambda) = \phi(2^{-k}|\lambda|) \lambda \tilde{m}(\lambda)$, similar to its usage in the
previous section.
\begin{equation}
T_k(x,y) = \frac{-1}{4\pi^2 i} \int_{-\infty}^{\infty} \int_{\R^3}
g_k(\lambda) R_V^+(\lambda^2)(x,z)V(z) \frac{e^{i\lambda|z-y|}}{|z-y|}\,dz d\lambda
\end{equation} 
For a fixed choice of $k$, the estimate $|R_V^+(\lambda^2)(x,z)| \les |x-z|^{-1}$
ensures that the integral over $\R^3 \times \R$ is absolutely convergent
(and bounded by $2^{2k}|x-y|^{-1}$), hence the order can be rearranged.
Applying Parseval's identity yields
\begin{equation*}
\int_{-\infty}^{\infty} g_k(\lambda) R_V^+(\lambda^2)(x,z)e^{i\lambda|z-y|}\,d\lambda
=  \int_{-\infty}^\infty \widecheck{g}_k(t+|z-y|) \mathcal{F}(R_V^+({\scriptstyle (\,\cdot\,)^2}))(t, x, z)\,dt.
\end{equation*}
Now the integral over $z \in \R^3$ is subdivided into two regions.  Let $\eta$ be a cutoff function on $[0,\infty)$
that is supported in $[0,\frac12)$ and identically 1 on the interval $[0,\frac14]$.  We may write
\begin{align*}
T_k(x,y) &= \frac{-1}{4\pi^2 i}\int_{\R^3} \eta\Big(\frac{|z-y|}{|x-y|}\Big) \frac{V(z)}{|z-y|}  
\int_{-\infty}^\infty \widecheck{g}_k(t+|z-y|) \mathcal{F}(R_V^+({\scriptstyle (\,\cdot\,)^2}))(t, x, z)\,dt dz \\
 &\ + \frac{-1}{4\pi^2 i}\int_{\R^3}\Big(1 -  \eta\Big(\frac{|z-y|}{|x-y|}\Big)\Big) \frac{V(z)}{|z-y|}  
\int_{-\infty}^\infty \widecheck{g}_k(t+|z-y|) \mathcal{F}(R_V^+({\scriptstyle (\,\cdot\,)^2}))(t, x, z)\,dt dz \\
&= T_{k,1}(x,y) + T_{k,2}(x,y).
\end{align*}

For both $T_{k,1}$ and $T_{k,2}$ the next order of business is to obtain bounds on the integral over $t$.
Using the structure of $\mathcal{F}(R_V^+)$ from~\eqref{eq:goal} and Proposition~\ref{prop:cone}, this
integral splits neatly into the intervals $t < |x-z|$ and $t \geq |x-z|$.
On the right halfline, the pointwise bound $|\widecheck{g_k}(t)| \leq C2^{2k}(1+2^k|t|)^{-s}$ combines with the
integral bound in~\eqref{eq:goal} to yield
$$
\Big| \int_{|x-z|}^\infty \widecheck{g}_k(t+|z-y|)\mathcal{F}(R_V^+({\scriptstyle (\,\cdot\,)^2}))(t, x, z)\,dt \Big|
\les \frac{2^{2k}}{|x-z|(1+2^{k}(|x-z| + |z-y|))^s}.
$$ 
On the left half-line we have
\begin{multline*}
 \int_{-\infty}^{|x-z|} \widecheck{g}_k(t+|z-y|)\mathcal{F}(R_V^+({\scriptstyle (\,\cdot\,)^2}))(t, x, z)\,dt  \\
\begin{aligned}
&= -\pi \sum_{\mu_j} \frac{P_j(x,z)}{\mu_j} \int_{-\infty}^{|x-z|}\widecheck{g}_k(t+|z-y|) e^{\mu_j t}\,dt \\
&= -\pi \sum_{\mu_j} \frac{P_j(x,z)e^{|x-z|\mu_j}}{\mu_j} \int_{-\infty}^\infty \frac{g_k(\lambda)}{\lambda - i\mu_j}
e^{i(|x-z| + |z-y|)\lambda}\,d\lambda.
\end{aligned}
\end{multline*}
As discussed previously, the $H^s$ norm of $(\lambda - i\mu_j)^{-1}g_k$ is controlled by $\mu_j^{-1} \|g_k\|_{H^s}$,
so the integral over $\lambda$ will be bounded by $2^{2k}(1 + 2^{k}(|x-z|+|z-y|))^{-s}$.  Applying the projection
bound from Lemma~\ref{lem:projections} gives us
$$
\Big| \int_{-\infty}^{|x-z|} \widecheck{g}_k(t+|z-y|)\mathcal{F}(R_V^+({\scriptstyle (\,\cdot\,)^2}))(t, x, z)\,dt \Big|
\les \frac{2^{2k}}{|x-z|(1+2^{k}(|x-z| + |z-y|))^s}.
$$ 

This is already sufficient to obtain the desired control of the operators $T_{k,1}$.  Within the support of
$\eta\big(\frac{|z-y|}{|x-y|}\big)$, both $|x-z|$ and $|x-z|+ |z-y|$ are comparable in size to $|x-y|$, so
\begin{align*}
\sum_{k=-\infty}^\infty |T_{k,1}(x,y)| &\les \sum_{k=-\infty}^\infty \int_{\R^3} \eta\Big(\frac{|z-y|}{|x-y|}\Big) 
\frac{2^{2k}|V(z)|}{|z-y|\,|x-y| (1+2^k|x-y|)^s}\,dz \\
&\les \int_{\R^3} \eta\Big(\frac{|z-y|}{|x-y|}\Big) \frac{|V(z)|}{|z-y|\,|x-y|^3}\,dz \\
&\leq \frac{1}{|x-y|^3} \| \chi(|z-x| \sim |x-y|)V(z)\|_{\K_z}.
\end{align*}
The assumption $s>2$ is used to gain convergence of the series for large $k$.  It follows that for any $f \in L^1(\R^3)$,
\begin{equation}
\begin{aligned}
\sum_{k=-\infty}^\infty |T_{k,1}f(x)| &\les \sum_{\ell = -\infty}^\infty \|\chi(|z-x| \sim 2^\ell)V(z)\|_{\K_z}
\int_{|x-y| \sim 2^\ell} \frac{|f(y)|}{2^{3\ell}}\,dy \\
&\les \Big(\sum_{\ell = -\infty}^\infty \|\chi(|z-x| \sim 2^\ell)V(z)\|_{\K_z}\Big) Mf(x),
\end{aligned}
\end{equation}
where $Mf(x)$ is the (centered) Hardy-Littlewood maximal function of $f$.  The sum-total of contributions of $T_{k,1}f(x)$
to $m(\sqrt{H})f(x)$ satisfies a weak-(1,1) bound, as well as $L^p$ bounds for $1<p<\infty$.

The contributions of $T_{k,2}$ are a standard Calder\'on-Zygmund operator.
In the support of $1 - \eta\big(\frac{|z-y|}{|x-y|}\big)$, the size of $|z-y|$ is bounded below by $\frac14|x-y|$.
 We have the pointwise bound
\begin{equation} \label{eq:T_k2}
\begin{aligned}
|T_{k,2}(x,y)| &\les  \int_{\R^3} \Big(1- \eta\Big(\frac{|z-y|}{|x-y|}\Big)\Big) 
\frac{2^{2k}|V(z)|}{|z-y|\,|x-z| (1+2^k(|x-z| +|z-y|)^s}\,dz \\
&\les \int_{|z-y| > \frac14|x-y|} \frac{2^{2k}|V(z)|}{|x-y|\,|x-z|(1+2^k|x-y|)^s}\,dz \\
&\les \frac{2^{2k}}{|x-y|(1+2^k|x-y|)^s}\|V\|_\K.
\end{aligned}
\end{equation}
and consequently
\begin{equation*}
\sum_{k=-\infty}^\infty |T_{k,2}(x,y)| \leq \frac{1}{|x-y|^3} \| V\|_{\K}.
\end{equation*}

Additionally, since the support of $g_k$ lies where $|\lambda| \sim 2^k$,
the $H^s$ norms of $\lambda g_k(\lambda)$ are controlled by $2^k$ times the corresponding norm of $g_k$.
One can repeat the same arguments as before to conclude that
$$
\Big|\int_{-\infty}^\infty \widecheck{g}_k'(t+|z-y|) \mathcal{F}(R_V^+({\scriptstyle (\,\cdot\,)^2}))(t, x, z)\,dt \Big|
\les \frac{2^{3k}}{|x-z|(1+2^{k}(|x-z| + |z-y|))^s}.
$$
Meanwhile, the gradient of $\big(1-\eta\big(\frac{|z-y|}{|x-y|}\big)\big)|z-y|^{-1}$ is supported
where $|z-y| \geq \frac14|x-y|$ and is bounded above by a multiple of $|x-y|^{-2}$.  Applying the Leibniz rule to the
integral defining $T_{k,2}(x,y)$ yields
\begin{align*}
|\nabla_y T_{k,2}(x,y)| &\les \int_{|z-y| > \frac14|x-y|} \frac{2^{2k}|V(z)|}{|x-y|\,|x-z|(1+2^k|x-y|)^s}
\Big(2^k + \frac{1}{|x-y|}\Big)\,dz \\
&\les \frac{2^{2k}}{|x-y|(1+2^k|x-y|)^s}\Big( 2^k + \frac{1}{|x-y|}\Big) \|V\|_\K.
\end{align*}

Thus for two points $y$, $y_0$ with $|y-y_0| < \frac12 |x-y|$ we can conclude that
\begin{equation}
|T_{k,2}(x,y) - T_{k,2}(x,y_0)| \les \frac{2^{2k}\|V\|_\K}{|x-y|(1+2^k|x-y|)^s}\min(2^k |y-y_0|, 1)
\end{equation}
where the second option in the minimum comes from the pointwise bound in~\eqref{eq:T_k2}.
If $2 < s < 3$ then
$$
\sum_{k=-\infty}^\infty |T_{k,2}(x,y) - T_{k,2}(x,y_0)| \les \frac{|y-y_0|^{s-2}}{|x-y|^{s+1}}
$$
and if $s>3$ then $\sum_k |\nabla_y T_{k,2}(x,y)| \les |x-y|^{-4}$.

Now that the kernel of $\sum_k T_{k,2}$ meets the conditions for a Calder\'on-Zygmund operator, 
we must go back and check that it also satisfies an {\it a priori} bound on $L^2(\R^3)$.  In lieu of
a direct proof, we note that
$$
\sum_{k=-\infty}^\infty T_{k,2} = m(\sqrt{H}) - m(\sqrt{-\Delta}) - \sum_{k=-\infty}^\infty T_{k,1}.
$$
The first two operators on the right side are $L^2$-bounded due to the boundedness of $m$,
and the final sum is $L^2$-bounded because it is dominated by the Hardy-Littlewood maximal function.

Therefore $\sum_k T_{k,2}$ is of weak-$(1, 1)$ type and a bounded operator on $L^p$, $1<p< \infty$.
\end{proof}

\section{Negative powers of the Hamiltonian} \lb{sec:Fractional}
In this section we extend Theorem~\ref{thm:multiplier3} to a
class of multipliers that are roughly homogeneous of order $-\alpha \in (-3, 0)$. This result does not follow from Theorem~\ref{thm:multiplier3}, though the proof uses similar techniques.

\begin{theorem}[Generalized fractional integration] \label{thm:multiplier4}
 Assume $V \in \K_0 \cap L^{3/2, \infty}$, and $H = -\Delta + V$ has no eigenvalue or resonance at
zero, and no positive eigenvalues.  Let $\phi$ be a $C^\infty$ function supported on
$[\frac12, 2]$ such that $\{\phi(2^{-k}\lambda)\}_{k\in\mathbb Z}$ forms a partition of unity for
$\R^+$.

Let $\alpha \in (0,3)$ and suppose $m: (0,\infty) \to \mathbb C$ satisfies
\begin{equation} \label{eq:HsAlphaCondition}
\sup_k  \|\phi(\lambda)m(2^k\lambda)\|_{H^s} = M < \infty
\end{equation}
for some  $s> \max(\frac32 - \alpha, 1 - \frac23 \alpha)$.
Then $H^{-\alpha/2}m(\sqrt{H}) := |H|^{-\alpha/2} m(\sqrt{H P_c})$ is an integral operator that is bounded from
$L^p(\R^3)$ to $L^q(\R^3)$ for all pairs with $\frac1p - \frac1q = \frac{\alpha}3$ and $1 < p < \frac{3}{\alpha}$, 
with operator norm less than $C_{\alpha,s,p} M$.

If $V \in \mc K_0$ and~\eqref{eq:HsAlphaCondition} holds for some $s>2-\alpha$, then $H^{-\alpha/2}m(\sqrt{H})$ also has an integral kernel 
$K(x,y)$ that is bounded pointwise by
\begin{equation}
|K(x,y)| \les \frac{M}{|x-y|^{3-\alpha}}.
\end{equation}
\end{theorem}

\begin{observation}
Using the trivial choice $m(\lambda) \equiv 1$, this shows that the kernel of $H^{-\alpha/2}P_c$ is bounded by $|x-y|^{\alpha-3}$.  Since each eigenfunction of $H$ is bounded and exponentially decreasing, and there are finitely many of them, the same is true for $H^{-\alpha/2}$, $0< \alpha < 3$. 
\end{observation}

\begin{observation}
The range of $s$ is nearly sharp.  There are counterexamples of Bochner-Riesz type if $s < \frac32 - \alpha$.  There are counterexamples of the form $m(\lambda) = N^{-s}\phi(\lambda)\sin(N\lambda)$ if $s < 1-\frac23 \alpha$.
\end{observation}

\begin{observation}
When $\alpha \geq 2$ the multiplier $|\lambda|^{-\alpha/2}m(\sqrt{\lambda})$ is not locally integrable at zero.  In these cases we define $H^{-\alpha/2}m(\sqrt{H})$ as the sum $\sum_{k \in \Z} \phi(2^{-k}\sqrt{H})|H|^{-\alpha/2}m(\sqrt{H})$, which converges in the weak operator topology to a limit
independent of the choice of partition of unity $\phi$.
\end{observation}

\begin{proof}
The pointwise bound follows the same argument as in Theorem~\ref{thm:multiplier3}.
For each $k \in \Z$, let $g_k(\lambda) = \phi(2^{-k}\lambda)\lambda |\lambda|^{-\alpha}\tilde{m}(\lambda)$.
By assumption, $\phi(\lambda)\lambda |\lambda|^{-\alpha} \tilde{m}(2^{k}\lambda)$ is a compactly supported
function in $H^s$.  It is also the derivative of a function in $H^{s+1}$ with support in the interval $[-2^{k+2}, 2^{k+2}]$, 
thanks to the fact that $\lambda \tilde{m}$ is an odd function.

Thus the Fourier transform of $\phi(\lambda)\lambda |\lambda|^{-\alpha}\tilde{m}(2^{k}\lambda)$ is bounded by
$|t|(1+|t|)^{-(s+1)}M$, hence the Fourier transform of $g_k(\lambda)$
is bounded by $2^{(3-\alpha)k}|t|(1+2^k|t|)^{-(s+1)}M$.  Assuming that $\alpha < 3$ and $s > 2 - \alpha$, it follows that
\begin{equation} \lb{mhatalpha}
\sum_{k=-\infty}^\infty \big|\big(\phi(2^{-k}\lambda) \lambda |\lambda|^{-\alpha}\tilde{m}\big)\widecheck{\phantom{i}}(t) \big| \les M|t|^{\alpha - 2}.
\end{equation}

From here the argument follows~\eqref{eq:Stonemultiplier} and its discussion.  The kernel of $H^{-\alpha/2}m(\sqrt{H})$ is given by
\begin{equation*}
K(x,y) = (\pi i)^{-1}\int_{-\infty}^\infty \sum_{k=-\infty}^\infty (\phi(2^{-k}\lambda) \lambda |\lambda|^{-\alpha}\tilde{m})\widecheck{\phantom{i}}(t) 
\mathcal F R_V^+({\scriptstyle (\,\cdot\,)^2})(t,x,y)\,dt.
\end{equation*}
Applying Lemma~\ref{lem:LpTransfer} with $p= \infty$ and $\beta = 3-\alpha$ shows that the portion of the integral where
$t > |x-y|$ is bounded by $M|x-y|^{\alpha-3}$.  The integral over $t < |x-y|$ consists of a finite sum of terms, each of which is bounded
(as in~\eqref{eq:K2bound}) by
\begin{equation*}
\frac{1}{|x-y|}
\Big|\int_\R \lambda |\lambda|^{-\alpha} \Big(\frac{\tilde{m}(\lambda)}{\lambda - i\mu_j}\Big)e^{i\lambda|x-y|}\,d\lambda\Big|
\les \frac{M}{|x-y|^{3-\alpha}},
\end{equation*}
with the last inequality taking advantage of~\eqref{mhatalpha} and the fact that $\frac{\tilde{m}(\lambda)}{\lambda - i\mu_j}$
satisfies~\eqref{eq:HsAlphaCondition} for the same range of $s$ as $\tilde{m}$.

The $L^p \to L^q$ estimates are established by complex interpolation of operators. 
The dyadic structure of $m(\lambda)$ is a bit unwieldy here, so we define
$n(\rho) = m(2^{\rho})$ and set
\begin{equation*}
m_z(\lambda) = \Big(1 - \frac{\partial^2}{\partial\rho^2}\Big)^{L(z)}n(\log_2 \lambda),
\end{equation*}
where $L(z)$ is a linear function with real coefficients.
Note that for any given $s$, the $H^s$ norm of $n(\rho)$ on an interval $[k, k+1]$ is comparable to the $H^s$ norm of
$m(2^{k}\lambda)$ on the interval $[1, 2]$, thus the effect of $\frac{\partial}{\partial\rho}$ on $n(\rho)$ is to modulate
the Sobolev space condition in~\eqref{eq:HsAlphaCondition}.  

Interpolation will be performed on the family of operators
$H^{-z/2}m_z(\sqrt{H})$, with the linear function $L(z)$ chosen according to the available endpoint estimates.  Note that the multiplier is modified in two ways by the imaginary part of $z$.  Differentiation to imaginary order in $(1 - \frac{\partial^2}{\partial \rho^2})^{L(z)}$ preserves the $H^s$ norms of $m_z(\lambda)$.  In the operator $H^{-z/2}$, the imaginary part of $z$ corresponds to multiplication of $m_z(\lambda)$ by $|\lambda|^{i( \Im(z))}$.  This
acts boundedly on $H^s$  with an operator norm that grows like $\langle \Im(z)\rangle^{|s|}$ as the imaginary part increases.

When the real part of $z$ is zero and $s > \frac32$, the operators $H^{-z/2}m(\sqrt{H})$ fit within the framework of Theorem~\ref{thm:multiplier3}. They are bounded on every $L^p(\R^3)$, $1 < p < \infty$.

When the real part of $z$ is $3-\epsilon$ and $s > -1 + \epsilon = 2 - \Re(z)$, we have just shown that  $H^{-z/2}m(\sqrt{H})$ has integral kernel bounded by $|x-y|^{-\epsilon}$.  By the Hardy-Littlewood-Sobolev inequality it maps $L^p(\R^3)$ to $L^q(\R^3)$ for pairs $1 < p,q < \infty$ with $\frac{3}{p} - \frac{3}{q} = 3-\epsilon$.

The third case where we can establish an {\it a priori} bound is the real part of $z$ is $\frac32$ and $s>0$.  This does not follow directly from previous estimates, however the basic tools are the same.  The kernel of $H^{-3/4}m(\sqrt{H})$ is given by
\begin{equation*}
K(x,y) = (\pi i)^{-1}\int_{-\infty}^\infty \sum_{k=-\infty}^\infty (\phi(2^{-k}\lambda) \lambda |\lambda|^{-\frac32}\tilde{m})\widecheck{\phantom{i}}(t) 
\mathcal F R_V^+({\scriptstyle (\,\cdot\,)^2})(t,x,y)\,dt.
\end{equation*}
We would like to show that for each $y \in \R^3$, $K(\,\cdot\, , y)$ belongs to the Lorentz space $L^{2,\infty}(\R^3)$ with a uniform bound on its norm.  This is accomplished via the following claim which will be proved afterwards.
\begin{lemma} \label{lem:weakL2}
For every $y \in \R^3$ and $k_0 \in \Z$, there is a decomposition $K(x,y)  = F(x,y) + G(x,y)$ such that
\begin{equation*}
\|F(\, \cdot\, , y)\|_{L^2(|x-y| > 2^{k_0})} \leq C \quad \text{ and } \quad 
\|G(\,\cdot\, , y)\|_{L^\infty(|x-y| > 2^{k_0})} \leq C 2^{-\frac32 k_0}.
\end{equation*}
The constant $C < \infty$ does not depend on $y$ or $k_0$.
\end{lemma}
 Assuming this result for now, it follows that the set $\{x \in \R^3: |K(x,y)| > 2C 2^{-3k_0/2}\}$ is contained in the ball of radius $2^{k_0}$ centered at $y$ together with the set where $|F(x,y)| > C 2^{-3k_0/2}$.  Their union has measure bounded by $2^{3k_0}$.

The Minkowski inequality then shows that $H^{-3/4}m(\sqrt{H})$ maps $L^1(\R^3)$ to $L^{2,\infty}(\R^3)$.  The dual statement is true as well, and the intermediate $L^p \to L^q$ bounds follow by Marcinkiewicz interpolation.

Interpolating between the three cases described above shows that $H^{-\alpha/2}m(\sqrt{H})$ has the same $L^p \to L^q$ mapping properties as the fractional integrals $(-\Delta)^{-\alpha/2}$ whenever $0 < \alpha < 3$ and the pair $(\alpha, s)$ lies above the convex hull of the points $(0, \frac32)$, $(3, -1)$, and $(\frac32, 0)$.
\end{proof}

\begin{proof}[Proof of Lemma~\ref{lem:weakL2}]
Once again we use Lemma~\ref{lem:LpTransfer} to obtain the desired bounds on $F(x,y)$ and $G(x,y)$.  Let $g_k(\lambda) = \phi(2^{-k}\lambda)\lambda |\lambda|^{-3/2}\tilde{m}(\lambda)$.  The two main estimates on $\widecheck{g}_k(t)$ are
\begin{equation*}
|\widecheck{g}_k(t)| \les 2^{k/2} \quad \text{ and } \quad 
\big\|\widecheck{g}_k(\,\cdot\,)\langle 2^{k}\,\cdot\,\rangle^s\big\|_{L^2(\R)} \les 1.
\end{equation*}
It will turn out that the desired decomposition is given by
\begin{align*}
F(x,y) &= (\pi i)^{-1}\int_{-\infty}^\infty \sum_{k=1-k_0}^\infty \widecheck{g}_k(t) 
\mathcal F R_V^+({\scriptstyle (\,\cdot\,)^2})(t,x,y)\,dt, \\
G(x,y) &= (\pi i)^{-1}\int_{-\infty}^\infty \sum_{k=-\infty}^{-k_0}\widecheck{g}_k(t) 
\mathcal F R_V^+({\scriptstyle (\,\cdot\,)^2})(t,x,y)\,dt.
\end{align*}

If $k > -k_0$, then the $L^2(\R)$ estmate for $\widecheck{g}_k$ yields 
$\|\widecheck{g}_k\chi_{|t| > 2^{k_0}}\|_{L^2(\R)} \les 2^{-(k+k_0)s}$. Hence
\begin{equation*}
\Big\|\sum_{k > -k_0} \widecheck{g}_k \chi_{t > 2^{k_0}}\Big\|_{L^2(\R)} \les 1.
\end{equation*}
Lemma~\ref{lem:LpTransfer} then shows that the portion of $F(x,y)$ derived from integrating over $|x-y| \leq t < \infty$ is bounded in $L^2(|x-y| > 2^{k_0})$.  The integral over $t < |x-y|$ once again consists of a finite sum of terms, each of which is bounded by
\begin{equation*}
\frac{1}{|x-y|}
\Big|\int_\R \lambda |\lambda|^{-\alpha} \Big(\frac{\tilde{m}(\lambda)}{\lambda - i\mu_j}\Big)e^{i\lambda|x-y|}\,d\lambda\Big|.
\end{equation*}
We note that $\frac{m}{\lambda - i\mu_j}$ satisfies the bounds in~\eqref{eq:HsAlphaCondition} that $m(\lambda)$ does.
It follows that this part of $F(x,y)$ has $L^2$ norm controlled by
\begin{align*}
\Big\| |x-y|^{-1} \sum_{k> -k_0} \Big(\frac{g_k}{\lambda - i\mu_j}\Big)\widecheck{\phantom{I}}(|x-y|)\Big\|_{L^2(|x-y| > 2^{k_0})}
&= \sqrt{4\pi} \Big\| \sum_{k> -k_0} \Big(\frac{g_k}{\lambda - i\mu_j}\Big)\widecheck{\phantom{I}}(t)\Big\|_{L^2(t > 2^{k_0})} 
\\ &\les 1.
\end{align*}

The calculations for $G(x,y)$ follow the same pattern, except that the $L^\infty$ bound is used for $\widecheck{g}$ instead.  For
$k \leq -k_0$, $|\widecheck{g}_k(t)| \les 2^{k/2}$, which makes $\| \sum_{k \leq -k_0} \widecheck{g}_k \|_{L^\infty(\R)} \les 2^{-k_0/2}$.
Applying Lemma~\ref{lem:LpTransfer} with $p=\infty$, $\beta = 1$ shows that the portion of $G(x,y)$ derived from integrating over $|x-y| \leq t < \infty$ is bounded pointwise by $|x-y|^{-1}2^{-k_0/2}$.  The same bounds on the (inverse) Fourier transform of $\frac{g_k}{\lambda - i\mu_j}$ lead to an identical bound on the remaining part of $G(x,y)$.

It follows that for all $|x-y| > 2^{k_0}$, $|G(x,y)| \les 2^{-3k_0/2}$.
\end{proof}

\section{Proof of Strichartz estimates} \lb{sec:Strichartz}

\subsection{Auxiliary results}
Our proof of Strichartz estimates follows the one given by Beals in \cite{beals} and uses complex interpolation instead of the more usual Paley-Littlewood dyadic decomposition. It depends on the following two ingredients:\\
1. An $L^1$ to $L^\infty$ decay estimate for the perturbed wave evolution\\
2. $L^p$ bounds for $(-\Delta)^{i\sigma}$ and $H^{i\sigma}$, $1<p<\infty$.\\
Notably, it does not use the square function or localization in frequency for $H$.

\begin{lemma} For $s_1, s_2 \geq 0$, $0 \leq s = s_1 + s_2 < 1$,
\be\lb{Delta}
\|(-\Delta)^{-s_1} \cos(t \sqrt H) P_c (-\Delta)^{-s_2} f\|_{L^{\frac 2 {1-s}}} \les |t|^{-s} \|f\|_{L^{\frac 2 {1+s}}}.
\ee
\end{lemma}
The endpoint $s=1$ estimates also hold, after replacing $L^1$ by $\mc H^1$ and $L^\infty$ by $\BMO$.
\begin{proof} Start from the fact, proved in \cite{becgol}, that
$$
\|\cos(t \sqrt H) P_c f\|_{L^\infty} \les |t|^{-1} \|\Delta f\|_{L^1},
$$
which we rewrite as
\be\lb{cosh}
\|\cos(t\sqrt H) P_c (-\Delta)^{-1} f\|_{L^\infty} \les |t|^{-1} \|f\|_{L^1}.
\ee
By taking the adjoint, we also see that
$$
\|(-\Delta)^{-1} \cos(t\sqrt H) P_c f\|_{L^\infty} \les |t|^{-1} \|f\|_{L^1}.
$$
Since $(-\Delta)^{i\sigma}$ is bounded on the Hardy space $\mc H^1$ and on $\BMO$ with norms of appropriate growth, we can use complex interpolation between
$$
\|(-\Delta)^{i\sigma_1} \cos(t\sqrt H) P_c {(-\Delta)^{-1+i\sigma_2}} f\|_{\BMO} \les |t|^{-1} \|f\|_{\mc H^1},
$$
$$
\|{(-\Delta)^{-1+i\sigma_1}} \cos(t\sqrt H) P_c (-\Delta)^{i\sigma_2} f\|_{BMO} \les |t|^{-1} \|f\|_{\mc H^1},
$$
and
$$
\|(-\Delta)^{i\sigma_1} \cos(t\sqrt H) P_c (-\Delta)^{i\sigma_2}\|_{L^2} \les \|f\|_{L^2}
$$
to prove that (\ref{Delta}) holds.
\end{proof}

The next step is replacing $(-\Delta)^{-s}$ by $H^{-s} P_c$ in this inequality. We can do it if $(-\Delta)^sP_c H^{-s}$ is a bounded operator on $L^{\frac{2}{1+s}}(\R^3)$.  This property was established in~\cite{hong}. 
\begin{lemma}[{\cite[Lemma 1.4]{hong}}] \lb{intertwine} Suppose $V \in\K_0 \cap L^{3/2,\infty}(\R^3)$ and $H = -\Delta+ V$ has no eigenvalue or resonance at zero, and no positive eigenvalues. For every $0 \leq s \leq 1$, 
\begin{enumerate}
\item Operators  $P_c H^s (-\Delta)^{-s},\;(-\Delta)^s H^{-s} P_c$ belong to $\B(L^p)$ over the range $1<p<\frac 3 {2s}$.
\item Operators $(-\Delta)^{-s} H^s P_c,\;P_c H^{-s} (-\Delta)^s$ belong to $\B(L^p)$ in the range $\frac 3 {3-2s}<p<\infty$.
\end{enumerate}

Moreover, $P_c H (-\Delta)^{-1},\;\Delta H^{-1} P_c \in \B(L^1)$ and $(-\Delta)^{-1} H P_c,\;P_c H^{-1} \Delta \in \B(L^\infty)$.
\end{lemma}

The proof of this lemma is based on Mihlin multiplier bounds, both for $-\Delta$ and for $H$, and complex interpolation.

\begin{observation}
The assumption $V \in L^{3/2,\infty}$ is essential for establishing the $s=1$ cases of Lemma~\ref{intertwine}.  For general $V \in \K_0$, we conjecture that the optimal range of $L^p$ bounds for $(-\Delta)^s H^{-s}P_c$ is
$1 < p < \frac{1}{s}$ with equality sometimes permitted.  This still includes the relevant exponent $p = \frac{2}{1+s}$. The case $s=\frac12$, $p=2$ is proved in~\cite{goldberg}.
\end{observation}

In particular, for $s<3/4$ all these intertwining operators are $L^2$ bounded. Thus, one can use powers of $H$ in order to characterize the Sobolev spaces $P_c \dot H^s$, $-3/2 < s < 3/2$ (and the negative energy component is trivial under our assumptions, since $P_c \in \B(L^1) \cap \B(L^\infty)$):
\begin{lemma} Suppose $V \in \K_0 \cap L^{3/2,\infty}(\R^3)$ and $H = -\Delta+ V$ has no eigenvalue or resonance at zero, and no positive eigenvalues. For $-3/2<s<3/2$,
\be\lb{new_sobolev}
\dot H^s = \{f \in \mc S': (-\Delta)^{s/2} f \in L^2\} = \{f \in \mc S': H^{s/2} f \in L^2\}.
\ee
More generally, the same is true for $\dot W^{s, p}$ if $1<p<3/s$ and $0\leq s \leq 2$ or if $\frac 3 {3-s}<p<\infty$ and $-2 \leq s \leq 0$ --- hence also for Besov spaces in this range.
\end{lemma}

\subsection{Strichartz estimates} Finally, we are in a position to prove the full range of Strichartz estimates for the wave equation in three dimensions.


For the free wave equation in $\R^{d+1}$
$$
u_{tt}-\Delta u = F,\ u(0)=u_0,\ u_t(0)=u_1,
$$
as proved by \cite{keeltao}, the following Strichartz estimates hold:
\be\lb{strichartz_est}
\|u\|_{L^\infty_t \dot H^s_x \cap L^p_t L^q_x} + \|u_t\|_{L^\infty_t \dot H^{s-1}_x} \les \|u_0\|_{\dot H^s} + \|u_1\|_{\dot H^{s-1}} + \|F\|_{L^{\tilde p'}_t L^{\tilde q'}_x},
\ee
where by scaling and translation invariance
\be\lb{condition1}
\frac 1 p + \frac d q = \frac d 2 - s = \frac 1 {\tilde p'} + \frac d {\tilde q'} - 2,\ 2 \leq p, q, \tilde p, \tilde q \leq \infty,
\ee
and the exponents must be wave-admissible:
\be\lb{condition2}
\frac 2 p + \frac {d-1} q \leq \frac {d-1} 2,\ \frac 2 {\tilde p} + \frac {d-1} {\tilde q} \leq \frac {d-1} 2.
\ee
The endpoint cases $(p, q) = (2, \infty)$ in $\R^{3+1}$ and $(p, q) = (4, \infty)$ in $\R^{2+1}$ are not true (and same for the inhomogeneous term, i.e.\ $(\tilde p, \tilde q) = (2, \infty)$ in $\R^3$).

In three dimensions, for wave-admissible exponents $(p, q)$, the pairs $(\frac 1 p, \frac 1 q)$ cover the triangle with vertices $(0, 0)$, $(0, \frac 1 2)$, and $(\frac 1 2, 0)$.

Thus, in $\R^{3+1}$, the homogeneous Strichartz estimates (i.e.\;$F=0$)
\be\lb{hom_str}
\|u\|_{L^p_t L^q_x} \les \|u_0\|_{\dot H^s} + \|u_1\|_{\dot H^{s-1}}
\ee
follow by interpolation between the segment $(\infty, \frac 6 {3-2s})$, $0 \leq s < 3/2$ --- trivial by $\dot H^s$ norm conservation --- and the sharp admissible segment $\frac 2 p + \frac {d-1} q = \frac {d-1} 2$, i.e.
\be\lb{sharp_str}
\frac 1 p + \frac 1 q = \frac 1 2.
\ee


Likewise, all the inhomogeneous estimates follow from $\dot H^s$ norm conservation --- along the line $(\tilde p', \tilde q')=(1,\frac 6 {3+2s})$, $0 \leq s < 3/2$ ---, plus the dual sharp admissible Strichartz inequalities
\be\lb{sharp_dual}
\|u\|_{L^\infty_t \dot H^s_x} \les \|F\|_{L^{\tilde p'}_t L^{\tilde q'}_x},\ \frac 1 {\tilde p} + \frac 1 {\tilde q} = \frac 1 2.
\ee
In turn, the latter are obtained by duality from the homogeneous sharp admissible Strichartz estimates (\ref{hom_str}-\ref{sharp_str}).

\begin{theorem}[Strichartz estimates] Provided that $V \in \K_0 \cap L^{3/2,\infty}(\R^3)$ and $H=-\Delta+V$ has no threshold bound states or embedded eigenvalues, the Strichartz estimates (\ref{strichartz_est}-\ref{condition2}) hold in $\R^{3+1}$ for the continuous spectrum projection $P_c$ of the solution $u$ of the equation
\be\lb{wave_V}
u_{tt} -\Delta u + V u = F,\ u(0)=u_0,\ u_t(0)=u_1.
\ee
\end{theorem}

\begin{proof}
Using the fact that
$$
\frac{\cos(t \sqrt H) P_c} {H^s} = \frac {(\cos(t \sqrt H) P_c} {(-\Delta)^{-s}} ((-\Delta)^s P_c H^{-s})
$$
and Lemma \ref{intertwine}, since $\frac 2 {1+s} < \frac 3 {2s}$, we immediately obtain from (\ref{Delta}) that for $0 \leq s < 1$
$$
\Big\|\frac{\cos(t\sqrt H) P_c} {H^s} f\Big\|_{L^{\frac 2 {1-s}}} \les |t|^{-s} \|f\|_{L^{\frac 2 {1+s}}}.
$$
The endpoint $s=1$ case is also true, having been proved in \cite{becgol}; however, using it in the interpolation would require a discussion of Hardy spaces for $H$, which we avoid here.

Therefore, by Young's inequality,
$$
\frac{\cos(t\sqrt H) P_c} {H^s} \in \B(L^{q_1, r}_t L^{\frac 2 {1+s}}_x, L^{q_2, r}_t L^{\frac 2 {1-s}}_x),
$$
where $\frac 1 {q_2} = \frac 1 {q_1} + s - 1$, with modifications at the endpoints: namely, when $q_1=1$ or $q_2=\infty$, then $L^{q_1, r}$ becomes $L^{q_1, 1}$ and $L^{q_2, r}$ becomes $L^{q_2, \infty}$.

In particular one can set the spaces to be dual to each other:
$$
\frac{\cos(t\sqrt H) P_c} {H^s} \in \B(L^{\frac 2 {2-s}, 2}_t L^{\frac 2 {1+s}}_x, L^{\frac 2 s, 2}_t L^{\frac 2 {1-s}}_x).
$$
By a $T T^*$ argument, then, for $0 \leq s < 1$
\be\lb{sharp_H}
\frac{\sin (t\sqrt H) P_c} {H^{s/2}},\ \frac{\cos (t\sqrt H) P_c} {H^{s/2}} \in \B(L^2_x, L^{\frac 2 s, 2}_t L^{\frac 2 {1-s}}_x).
\ee
This is the full range of sharply admissible homogeneous Strichartz estimates (\ref{sharp_str}). Due to Lemma \ref{intertwine}, powers of $H$ can be traded for powers of $-\Delta$ in this range.

Since for $-3/2<s<3/2$ the usual Sobolev norms are equivalent to those defined in terms of $H$, see (\ref{new_sobolev}), it follows that the positive energy part of equation (\ref{wave_V}) conserves the $\dot H^s$ norms, $-1/2<s<3/2$:
$$
\|P_c u\|_{L^\infty_t \dot H^s_x} \les \|u_0\|_{\dot H^s} + \|u_1\|_{\dot H^{s-1}}.
$$
This implies the homogeneous Strichartz estimates along the line segment $(p, q) = (\infty, \frac 6 {3-2s})$, $0 \leq s < 3/2$, by Sobolev embedding.

By interpolation with the sharp admissible estimates (\ref{sharp_H}), we obtain the full range of homogeneous Strichartz estimates (\ref{hom_str}).

Finally, the sharp admissible dual estimates (\ref{sharp_H}) imply by duality the sharp dual estimates (\ref{sharp_dual}). The other extreme case $(\tilde p', \tilde q') = (1, \frac 6 {3+2s})$, $0 \leq s < 3/2$, follows from the Sobolev norm conservation, as above. Interpolating between these two cases, we obtain all the inhomogeneous Strichartz estimates.
\end{proof}

\section*{Acknowledgments}
M.B.\;has been supported by the NSF grant DMS--1700293 and by the Simons Collaboration Grant No.\;429698.

M.G.\;is supported by the Simons Collaboration Grant No.\;635369.

The authors would like to thank Shijun Zheng for the helpful conversations.

\end{document}